\documentclass[10pt]{article}
\usepackage{charter}\usepackage[T1]{fontenc}\usepackage{textcomp}\usepackage{pifont}
\usepackage[scaled=.92]{helvet}
\usepackage{eulervm,euscript}
\usepackage{caption2}
\DeclareFontFamily{T1}{pzc}{}
\DeclareFontShape{T1}{pzc}{m}{it}{<-> s * [1.200] pzcmi}{}
\DeclareMathAlphabet{\nicemathcal}{OT1}{pzc}{m}{it}
\DeclareMathAlphabet{\oldmathcal}{OMS}{bch}{m}{n}
\usepackage{mathrsfs}
\usepackage{amsmath,amssymb,geometry,theorem,graphicx,wasysym,mathbbol}
\usepackage{pstricks,pst-node,pst-text,pst-tree,pstricks-add}

\newcommand{\postsectionheadinggap}{3ex}\newcommand{\presectionheadinggap}{4ex}
\makeatletter
\def\section{\@startsection{section}{1}{0pt}{-\presectionheadinggap plus -1ex minus 
-.2ex}{\postsectionheadinggap plus .2ex minus .3ex}{\normalfont\large\bf}}
\renewcommand\subsection{\@startsection{subsection}{2}{\z@}%
                                     {-3.25ex\@plus -1ex \@minus -.2ex}%
                                     {1.5ex \@plus .2ex}%
                                     {\normalfont\normalsize\bfseries}}
\makeatother

\usepackage{setspace}
\let\myCaption\caption
\renewcommand\caption[1]{%
\singlespacing
\myCaption{#1}}

\renewenvironment{abstract}
{\begin{quotation}\small
}{\end{quotation}}

\newcommand{\defn}[1]{{\textit{\textbf{#1}}}}
\newcommand{\myitem}[1]{\item[\textnormal{(#1)}]}

\newcommand{\ftnotemarknospace}[1]{\makebox[0ex]{$\mkern12mu{}^{\text{\footnotesize\textnormal #1}}$}}
\newcommand{\ftnotetextmark}[1]{{${}^{\text{\footnotesize\textnormal #1}}\mkern5mu$}}
\theoremheaderfont{\scshape}
\theorembodyfont{\normalfont\slshape}
\theoremstyle{plain}
\newtheorem{definition}{Definition}

\newtheorem{proposition}{Proposition}

\newtheorem{theorem}{Theorem}
\newtheorem{conjecture}{Conjecture}
\theorembodyfont{\normalfont}

\newenvironment{proof}{\begin{trivlist}\item{}\normalfont\textit{Proof.}}{\hfill$\square$\end{trivlist}}

\newcommand{\nth}{^{\text{th}}}

\newcommand{\ie}{\emph{i.e.}}
\newcommand{\eg}{\emph{e.g.}}
\newcommand{\cf}{\emph{cf.}}

\newcommand{\with}{\mkern2mu\&\mkern2mu}
\newcommand{\plus}{\raisebox{.6pt}{\makebox[1.8ex]{$\oplus$}}}

\newcommand{\tensor}{\raisebox{.6pt}{\makebox[1.8ex]{$\otimes$}}}
\newlength{\parrdp}\newlength{\parrht}
\newcommand{\parr}{\raisebox{-\parrdp}{\raisebox{\parrht}{\rotatebox{180}{$\&$}}}}
\newsavebox{\parrbox}\savebox{\parrbox}{\parr}
\newlength{\smallparrdp}\newlength{\smallparrht}
\newlength{\footnoteparrdp}\newlength{\footnoteparrht}
\newcommand{\footnoteparr}
{\mkern1mu\raisebox{-\footnoteparrdp}{\raisebox{\footnoteparrht}{\rotatebox{180}{\footnotesize$\&$}}}\mkern1mu}
\newlength{\scriptparrdp}\newlength{\scriptparrht}
\newlength{\tinyparrdp}\newlength{\tinyparrht}
\newdimen\arrayruleHwidth
\makeatletter
\def\Hline{\noalign{\ifnum0=`}\fi\hrule \@height \arrayruleHwidth
  \futurelet \@tempa\@xhline}
\makeatother

\newdimen\proofrulebreadth \proofrulebreadth=.05em
\newdimen\proofdotseparation \proofdotseparation=1.25ex
\newdimen\proofrulebaseline \proofrulebaseline=2ex
\newcount\proofdotnumber \proofdotnumber=3
\let\then\relax
\def\hfi{\hskip0pt plus.0001fil}
\mathchardef\squigto="3A3B
\newif\ifinsideprooftree\insideprooftreefalse
\newif\ifonleftofproofrule\onleftofproofrulefalse
\newif\ifproofdots\proofdotsfalse
\newif\ifdoubleproof\doubleprooffalse
\let\wereinproofbit\relax
\newdimen\shortenproofleft
\newdimen\shortenproofright
\newdimen\proofbelowshift
\newbox\proofabove
\newbox\proofbelow
\newbox\proofrulename
\def\shiftproofbelow{\let\next\relax\afterassignment\setshiftproofbelow\dimen0 }
\def\shiftproofbelowneg{\def\next{\multiply\dimen0 by-1 }%
\afterassignment\setshiftproofbelow\dimen0 }
\def\setshiftproofbelow{\next\proofbelowshift=\dimen0 }
\def\setproofrulebreadth{\proofrulebreadth}

\def\prooftree{
\ifnum  \lastpenalty=1
\then   \unpenalty
\else   \onleftofproofrulefalse
\fi
\ifonleftofproofrule
\else   \ifinsideprooftree
        \then   \hskip.5em plus1fil
        \fi
\fi
\bgroup
\setbox\proofbelow=\hbox{}\setbox\proofrulename=\hbox{}%
\let\justifies\proofover\let\leadsto\proofoverdots\let\Justifies\proofoverdbl
\let\using\proofusing\let\[\prooftree
\ifinsideprooftree\let\]\endprooftree\fi
\proofdotsfalse\doubleprooffalse
\let\thickness\setproofrulebreadth
\let\shiftright\shiftproofbelow \let\shift\shiftproofbelow
\let\shiftleft\shiftproofbelowneg
\let\ifwasinsideprooftree\ifinsideprooftree
\insideprooftreetrue
\setbox\proofabove=\hbox\bgroup$\displaystyle 
\let\wereinproofbit\prooftree
\shortenproofleft=0pt \shortenproofright=0pt \proofbelowshift=0pt
\onleftofproofruletrue\penalty1
}

\def\eproofbit{
\ifx    \wereinproofbit\prooftree
\then   \ifcase \lastpenalty
        \then   \shortenproofright=0pt  
        \or     \unpenalty\hfil         
        \or     \unpenalty\unskip       
        \else   \shortenproofright=0pt  
        \fi
\fi
\global\dimen0=\shortenproofleft
\global\dimen1=\shortenproofright
\global\dimen2=\proofrulebreadth
\global\dimen3=\proofbelowshift
\global\dimen4=\proofdotseparation
\global\count255=\proofdotnumber
$\egroup  
\shortenproofleft=\dimen0
\shortenproofright=\dimen1
\proofrulebreadth=\dimen2
\proofbelowshift=\dimen3
\proofdotseparation=\dimen4
\proofdotnumber=\count255
}

\def\proofover{
\eproofbit 
\setbox\proofbelow=\hbox\bgroup 
\let\wereinproofbit\proofover
$\displaystyle
}%
\def\proofoverdbl{
\eproofbit 
\doubleprooftrue
\setbox\proofbelow=\hbox\bgroup 
\let\wereinproofbit\proofoverdbl
$\displaystyle
}%
\def\proofoverdots{
\eproofbit 
\proofdotstrue
\setbox\proofbelow=\hbox\bgroup 
\let\wereinproofbit\proofoverdots
$\displaystyle
}%
\def\proofusing{
\eproofbit 
\setbox\proofrulename=\hbox\bgroup 
\let\wereinproofbit\proofusing
\kern0.3em$
}

\def\endprooftree{
\eproofbit 
  \dimen5 =0pt
\dimen0=\wd\proofabove \advance\dimen0-\shortenproofleft
\advance\dimen0-\shortenproofright
\dimen1=.5\dimen0 \advance\dimen1-.5\wd\proofbelow
\dimen4=\dimen1
\advance\dimen1\proofbelowshift \advance\dimen4-\proofbelowshift
\ifdim  \dimen1<0pt
\then   \advance\shortenproofleft\dimen1
        \advance\dimen0-\dimen1
        \dimen1=0pt
        \ifdim  \shortenproofleft<0pt
        \then   \setbox\proofabove=\hbox{%
                        \kern-\shortenproofleft\unhbox\proofabove}%
                \shortenproofleft=0pt
        \fi
\fi
\ifdim  \dimen4<0pt
\then   \advance\shortenproofright\dimen4
        \advance\dimen0-\dimen4
        \dimen4=0pt
\fi
\ifdim  \shortenproofright<\wd\proofrulename
\then   \shortenproofright=\wd\proofrulename
\fi
\dimen2=\shortenproofleft \advance\dimen2 by\dimen1
\dimen3=\shortenproofright\advance\dimen3 by\dimen4
\ifproofdots
\then
        \dimen6=\shortenproofleft \advance\dimen6 .5\dimen0
        \setbox1=\vbox to\proofdotseparation{\vss\hbox{$\cdot$}\vss}%
        \setbox0=\hbox{%
                \advance\dimen6-.5\wd1
                \kern\dimen6
                $\vcenter to\proofdotnumber\proofdotseparation
                        {\leaders\box1\vfill}$%
                \unhbox\proofrulename}%
\else   \dimen6=\fontdimen22\the\textfont2 
        \dimen7=\dimen6
        \advance\dimen6by.5\proofrulebreadth
        \advance\dimen7by-.5\proofrulebreadth
        \setbox0=\hbox{%
                \kern\shortenproofleft
                \ifdoubleproof
                \then   \hbox to\dimen0{%
                        $\mathsurround0pt\mathord=\mkern-6mu%
                        \cleaders\hbox{$\mkern-2mu=\mkern-2mu$}\hfill
                        \mkern-6mu\mathord=$}%
                \else   \vrule height\dimen6 depth-\dimen7 width\dimen0
                \fi
                \unhbox\proofrulename}%
        \ht0=\dimen6 \dp0=-\dimen7
\fi
\let\doll\relax
\ifwasinsideprooftree
\then   \let\VBOX\vbox
\else   \ifmmode\else$\let\doll=$\fi
        \let\VBOX\vcenter
\fi
\VBOX   {\baselineskip\proofrulebaseline \lineskip.2ex
        \expandafter\lineskiplimit\ifproofdots0ex\else-0.6ex\fi
        \hbox   spread\dimen5   {\hfi\unhbox\proofabove\hfi}%
        \hbox{\box0}%
        \hbox   {\kern\dimen2 \box\proofbelow}}\doll%
\global\dimen2=\dimen2
\global\dimen3=\dimen3
\egroup 
\ifonleftofproofrule
\then   \shortenproofleft=\dimen2
\fi
\shortenproofright=\dimen3
\onleftofproofrulefalse
\ifinsideprooftree
\then   \hskip.5em plus 1fil \penalty2
\fi
}

\newcommand{\edgesymb}{\raisebox{3pt}{\!\begin{psmatrix}[colsep=2.6ex]\rnode{l}{\rule{0pt}{1.2ex}}&\rnode{r}{\rule{0pt}{1.2ex}}%
\ncline[arrows=->,nodesep=2pt,arrowsize=2pt 1,arrowinset=.3,arrowlength=.5,linewidth=.3pt]{l}{r}\end{psmatrix}\!}}
\newcommand{\edge}{\mathrel{\edgesymb}}

\newcommand{\conflictstyle}{\psset{nodesep=1pt,linestyle=dotted,dotsep=1pt,linewidth=1pt}}
\newpsobject{conflictline}{ncline}{nodesep=1pt,linestyle=dotted,dotsep=1pt,linewidth=1pt}

\newcommand{\dual}[1]{\overline{#1}}

\newcommand{\perpp}{{}^{\perp}}

\newcommand{\girstcoh}{\mathrel{\raisebox{.3ex}{$\frown$}}}

\newcommand{\conflict}{\mkern1mu\#\mkern1mu}

\newcommand{\cohof}[1]{{#1}^{\#}}

\newcommand{\map}[2]{\Rnode{a}{}\hspace*{#2}\Rnode{b}{}\ncline[nodesep=3pt,offset=-.5pt]{->}{a}{b}\naput[labelsep=1pt]{#1}}
\newcommand{\umap}[2]{\Rnode{a}{}\hspace*{#2}\Rnode{b}{}\ncline[nodesep=3pt,offset=-.5pt]{->}{a}{b}\naput[labelsep=1pt]{#1}}

\newcommand{\skiprobin}[1]{}

\newcommand{\tagwith}[1]{\mkern2mu\with\mkern-6mu^{\raisebox{2pt}{\scriptsize{$#1$}}}\mkern-1mu}

\newcommand{\atomOne}{P}
\newcommand{\atomTwo}{Q}

\newcommand{\smallcircl}{\raisebox{-.05ex}{$\circ$}}
\newcommand{\smallsquar}{\raisebox{.4pt}[4.45pt][0pt]{\setlength\fboxrule{.38pt}\setlength\fboxsep{0pt}\framebox{\rule{1.52mm}{0mm}\rule{0mm}{1.5mm}}}}
\newcommand{\smallbulle}{\raisebox{-.09ex}{\small$\bullet$}}

\newcommand{\smallblacksquar}{\raisebox{.12ex}{\rule{1.5mm}{1.5mm}}}
\newcommand{\cross}{{\large\ding{55}}}
\newcommand{\chmark}{{\hspace*{-.3ex}{\large\bf\checkmark}\hspace*{-.3ex}}}
\newcommand{\openq}{{\large\bf?}}
\newcommand{\probno}{\cross\openq}

\newcommand{\lolly}{\multimap}
\newcommand{\leavesof}[1]{|#1|}

\newcommand{\adjacent}{\girstcoh}
\newcommand{\scriptadjacent}{\mbox{\tiny$\frown$}}
\newcommand{\cohmod}[1]{\textnormal{ [\textsf{mod} }#1\textnormal{]}}

\newcommand{\csym}{\textsf{C}}
\newcommand{\nodebox}[1]{\begin{picture}(7,7)\put(3.5,3.5){\makebox(0,0){#1}}\end{picture}}

\newcommand{\wirea}[2]{\naput[npos=#1,labelsep=2pt]{\makebox{\scriptsize$#2$}}}
\newcommand{\wireb}[2]{\nbput[npos=#1,labelsep=2pt]{\makebox{\scriptsize$#2$}}}

\newcommand{\contractionnodeanon}{\pscirclebox[framesep=0pt,linewidth=.1pt]{\nodebox{\csym}}}

\newcommand{\rulelabel}[1]{\mathsf{#1}}
\newcommand{\tensorlabel}{\otimes}
\newcommand{\parlabel}{\parr}

\newcommand{\axlabel}{\rulelabel{ax}}

\newcommand{\permlabel}{\rulelabel{\mkern1mu perm}}
\newcommand{\pluslabel}[1]{\plus_{#1}}

\newcommand{\withlabel}{\with}

\newcommand{\withOne}{\with}%
\newcommand{\withTwo}{\with\mkern-4mu'}%
\newcommand{\leafOne}{P}%
\newcommand{\leafTwo}{\leafOne}%
\newcommand{\leafThree}{\dual\leafOne}%
\newcommand{\leafFour}{\dual\leafOne}%
\newcommand{\leadsdownto}{\rput{-90}(0,0){$\leadsto$}}
\newcommand{\conflictgraph}[1]{{#1}^{\conflict}}
\newcommand{\adjacencygraph}[1]{{#1}^{^{\mkern-3mu\scriptadjacent}}}
\newcommand{\error}{\mathsf{E}}

\title{\LARGE
\textbf{Abstract p-time proof nets for MALL:\\ Conflict nets}
\author{\\[-2ex]
\large Dominic J.\ D.\ Hughes\thanks{Visiting Scholar, Computer Science Department, Stanford University, CA 94305.}
\\[1.5ex]
\normalsize Stanford University\\
\small January 11, 2007}\date{}
}

\begin{document}\thispagestyle{empty}
\maketitle

\begin{abstract}\vspace*{-2ex}\small
This paper presents proof nets for multiplicative-additive linear
logic (MALL), called \emph{conflict nets}.  They are \emph{efficient},
since both correctness and translation from a proof are p-time
(polynomial time), and \emph{abstract}, since they are invariant under
transposing adjacent $\with$-rules.

A conflict net on a sequent is concise: axiom links with a conflict
relation.
Conflict nets are a variant of (and were inspired by)
\emph{combinatorial proofs} introduced recently for classical logic:
each can be viewed as a maximal map (homomorphism) of contractible
coherence spaces ($P_4$-free graphs, or cographs), from axioms to
sequent.

The paper presents new results for other proof nets: (1) correctness
and cut elimination for slice nets (Hughes / van Glabbeek 2003) are
p-time, and (2) the cut elimination proposed for monomial nets (Girard
1996) does not work.
The subtleties which break monomial net cut elimination also apply to
conflict nets: 
as with monomial nets, existence of a confluent cut elimination
remains an open question.

\end{abstract}

\section{Introduction}

Jean-Yves Girard's seminal paper \cite{Gir87} on linear logic
introduced an elegant abstract representation of a proof called a
\emph{proof net}.
These original proof nets used \emph{boxes} \cite[p.\,45]{Gir87} to
deal with the superposition associated with $\with$-connectives.
Boxes mimic the sequent calculus $\with$-rule almost directly, so that
the following two proofs, which differ only in the order of adjacent
$\with$-rules, have distinct box nets:
\begin{center}%
\begin{prooftree}\thickness=.08em\label{intro-proofs}
\[
  \[
    \justifies
    \leafOne,\leafThree
  \]
  \hspace{2ex}
  \[
    \justifies
    \leafTwo,\leafThree
  \]
  \justifies
  \leafOne\withOne \leafTwo, \leafThree
  \using \withOne
\]
\hspace{4ex}
\[
  \[
    \justifies
    \leafOne,\leafFour
  \]
  \hspace{2ex}
  \[
    \justifies
    \leafTwo,\leafFour
  \]
  \justifies
  \leafOne\withOne \leafTwo, \leafFour
  \using \withOne
\]
\justifies
\leafOne\withOne \leafTwo, \leafThree\withTwo \leafFour
\using \withTwo
\end{prooftree}
\hspace{11ex}
\begin{prooftree}\thickness=.08em
\[
  \[
    \justifies
    \leafOne,\leafThree
  \]
  \hspace{2ex}
  \[
    \justifies
    \leafOne,\leafFour
  \]
  \justifies
  \leafOne, \leafThree\withTwo \leafFour
  \using \withTwo
\]
\hspace{4ex}
\[
  \[
    \justifies
    \leafTwo,\leafThree
  \]
  \hspace{2ex}
  \[
    \justifies
    \leafTwo,\leafFour
  \]
  \justifies
  \leafTwo, \leafThree\withTwo \leafFour
  \using \withTwo
\]
\justifies
\leafOne\withOne \leafTwo, \leafThree\withTwo \leafFour
\using \withOne
\end{prooftree}
\end{center}
(The marked connective $\with\mkern-4mu'$ is for distinction, we omit
sequent turnstiles $\vdash$, and $\dual P$ is the dual of $P$.)

The follow-up paper \cite{Gir96} tried a different approach to
superposition.  Every $\with$ is given an \emph{eigenvariable}, and
every node in the proof net has a list of possibly-negated
eigenvariables, its \emph{monomial}.
Monomial nets suffer
two main defects relative to box nets:
\begin{itemize}
\item
There is no canonical surjection from cut-free proofs to monomial
nets.\footnote{\label{note-monomial-box}There is a canonical
\emph{non}-surjective function: identify no formulas during
translation \cite[p.\,7]{Gir96}.  The image of this function is
precisely the box proof nets, disguised in monomial form.  So as a
semantics of cut-free proofs, this is exactly the box net semantics.
Since every box proof net is a monomial proof net, there are actually
\emph{more} monomial proof nets than box proof nets.}  One can no
longer ask ``Which proofs are identified upon translation to a proof
net?'': monomial nets fail to provide a semantics for cut-free
proofs.\footnote{See \cite{HG03,HG05} for a detailed explanation, with
examples.}
\item
Unfortunately the cut elimination proposed for monomial nets
\cite[p.\,24]{Gir87} does not work: Section~\ref{monomial-cut-elim}
gives a counterexample.  Existence of a confluent cut elimination
remains an open question.\footnote{One always has a trivial
non-confluent cut elimination via sequentialization, which is
uninteresting.}
\end{itemize}

The \emph{slice nets} of \cite{HG03,HG05} solve these problems by
taking a proof net to be a set of axiom linkings, or
\emph{slices}.\footnote{This underlying data structure is mentioned in
appendix A.1.6 of \cite{Gir96}.  The essential contribution of
\cite{HG03,HG05} was to provide the elusive geometric correctness criterion and exhibit a simple confluent cut elimination.}
(Equivalently, a slice net can be represented as a set of
boolean-weighted axiom links.)  There is a canonical surjection from
proofs.  For example, the two proofs above map to the following slice
net, comprising four axiom linkings, each linking containing just one
axiom link:\footnote{Note that this is not a single linking with four
axiom links; it is four linkings each with a single axiom link.  In
this particular case, there is also a canonical monomial net, but that
is not true in general.}
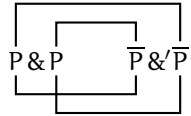
\begin{center}%
\vspace*{3ex}%
\newcommand{\gap}{\hspace{5ex}}%
\newcommand{\sequent}{\begin{math}%
\psset{unit=.9cm}%
\Rnode{a1}{\atomOne}
\with
\Rnode{a2}{\atomOne}
\gap
\Rnode{a1'}{\dual \atomOne}
\with\mkern-4mu'
\Rnode{a2'}{\dual \atomOne}
\end{math}}%
\sequent%
\ncbar[angle=90,nodesep=2pt,arm=.5cm]{a1}{a2'}
\ncbar[angle=90,nodesep=2pt,arm=.25cm]{a2}{a1'}
\ncbar[angle=-90,nodesep=2pt,arm=.25cm]{a1}{a1'}
\ncbar[angle=-90,nodesep=2pt,arm=.5cm]{a2}{a2'}
\vspace*{3ex}
\end{center}
Slice nets were shown to have a simple confluent cut elimination, and
a hyper-elimination which occurs independently slice-by-slice (by
GoI-style path composition), yielding a category
\cite{HG03,HG05}.
The present paper (Section~\ref{sec-slice-ptime}) proves that
correctness of slice nets is p-time.

But all is not rosy with slice nets: there can be an exponential
blowup in size when translating a proof.\footnote{Consider the unique
cut-free proof of $\tensor^n (1\with 1)$, where $\tensor^n$ denotes
iterated tensor $\tensor$ with $n$ arguments associated to the left
(\eg\ $\tensor^3 A\,=\,(A\tensor A)\tensor A\:$), in which
$\tensor$-rules are below $\with$-rules.
Since there are $n$ $\with$-rules, translating this proof $\Pi_n$ to
a slice net $\theta_n$ blows up exponentially: $\theta_n$ has $2^n$
slices (axiom linkings).  (For an example without the tensor unit $1$,
read each $1$ as $a\footnoteparr \dual a$.)  The exponential blowup
when mapping to a set of slices is mentioned in Appendix~A.1.6 of
\cite{Gir96}.}\,\footnote{\label{sleight}As remarked earlier, 
a set of slices can just as well be represented as a set of weighted
axiom links (arbitrary non-monomial boolean weights,
\eg\ $p\cup q$ for eigenvariables $p$ and $q$).
This trivial change of notation does not eliminate the exponential
blowup: with $n$ $\with$'s in the sequent, a boolean weight is a
subset of the $2^n$ hypercube.}
This is a flaw if we take seriously the
notion that a semantics is a \emph{structure preserving map}, or some
kind of \emph{homomorphism} from proofs: we are failing to respect
computational complexity.
A key insight of propositional proof complexity \cite{CR74} 
is that complexity is important in
decidable logics such as MALL.\footnote{In first-order logic,
which is undecidable, the value of a proof as a certificate of
theoremhood is absolutely clear.  But in a decidable, propositional
setting, what is the \emph{point} of being handed a proof?  To
determine theoremhood, we only need the formula.
The idea in propositional proof complexity is to reinstate and
quantify the value of a proof certificate: if the correctness of a
certificate can be checked in polynomial time in its size, and the
certificate is not `too big' relative to the formula, checking the
certificate will be faster than than deciding the theoremhood of the
formula.  See \cite{Urq95} for an accessible introduction to
propositional proof complexity.}%
\clearpage\noindent This paper presents a new notion of proof net, called a \emph{conflict
net}, such that:
\begin{itemize}
\myitem{1}
Checking correctness is p-time in the size of the proof
net.
\myitem{2} 
Translation from a proof is p-time (improving on slice nets \cite{HG03,HG05}).
\myitem{3}
Translation is invariant under transposing adjacent $\with$-rules, and
raising a $\parr$- or $\plus$-rule over a $\with$-rule (improving on
box nets \cite{Gir87} and monomial nets \cite{Gir96}\footnote{With
respect to the canonical (non-surjective) proof translation
\cite[p.\,7]{Gir96}.  See footnote~\ref{note-monomial-box}.}).
\myitem{4}
Extracting a sequentialization is p-time.
\myitem{5}
A conflict net on a sequent is concise: axiom links with a conflict
relation.
\myitem{6} 
Proof translation is simple: axioms become axiom links, and two
axiom links conflict iff they are from opposite branches above a
$\with$-rule.
\end{itemize}
\begin{figure}
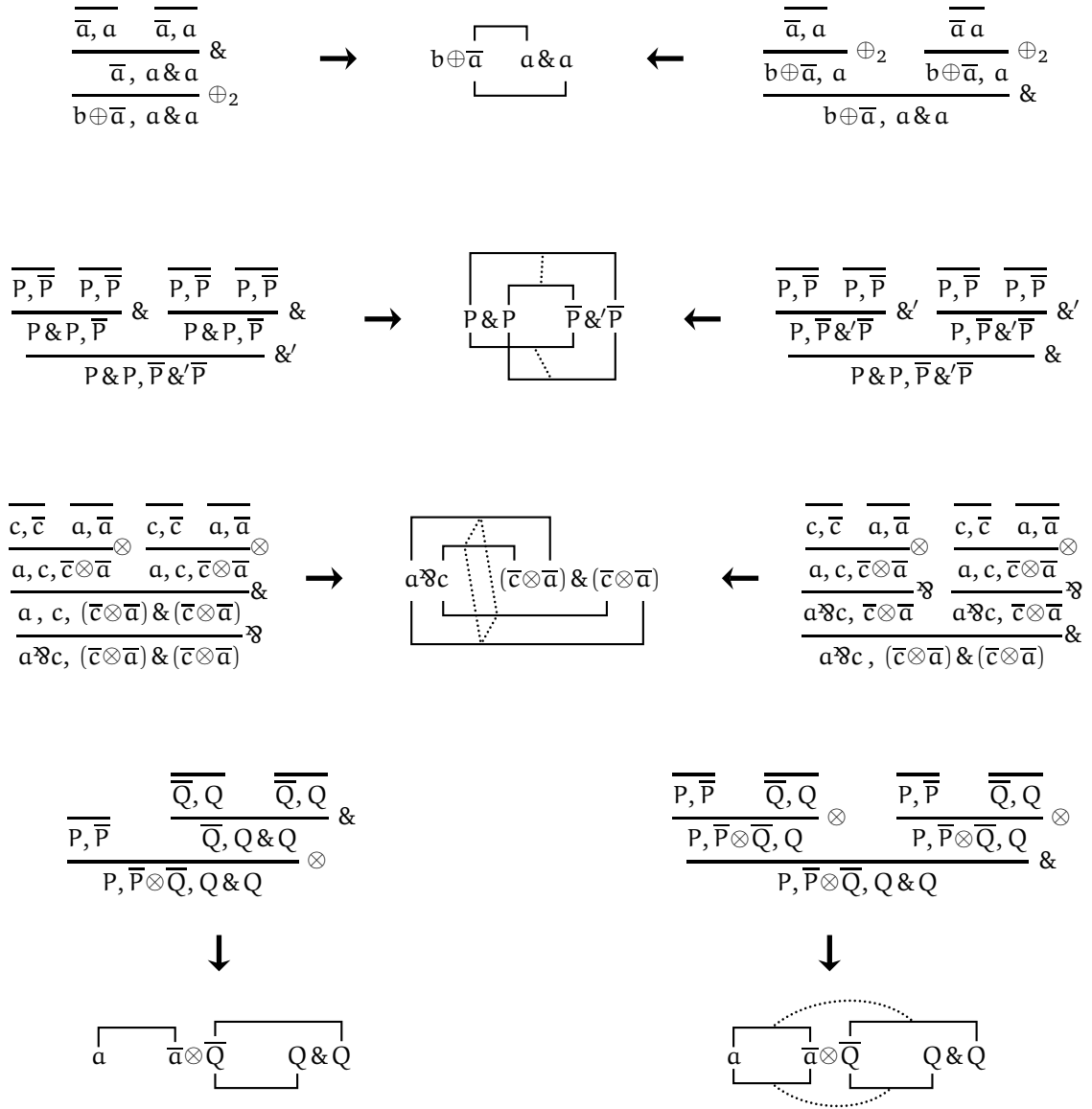
\vspace*{-2.2ex}
\begin{center}
\begin{prooftree}\thickness=.08em
\[
  \[ \justifies \dual a,a \] \hspace*{1ex} \[ \justifies \dual a,a \]
  \justifies \hspace*{3ex}\dual a\,,\,a\with a \using\with
\]
\justifies
b\plus \dual a\,,\,a\with a
\using\plus_2
\end{prooftree}
\begin{math}
\hspace*{6ex}\Rnode{m}{\strut}\hspace*{3ex}\Rnode{n}{\strut}\ncline[arrows=->,linewidth=2pt,arrowsize=9pt,arrowinset=.3,arrowlength=.5]{m}{n}\hspace*{6ex}
b\plus \Rnode{aa}{\strut\dual a}\hspace*{3ex}\Rnode{a}{\strut a}\with \rnode{a'}{\strut a}
\ncbar[angle=90,nodesep=2pt,arm=.2cm]{aa}{a}\ncput{\pnode{x}}
\ncbar[angle=-90,nodesep=2pt,arm=.2cm]{aa}{a'}\ncput{\pnode{y}}
\hspace*{6ex}\Rnode{m}{\strut}\hspace*{3ex}\Rnode{n}{\strut}\ncline[arrows=->,linewidth=2pt,arrowsize=9pt,arrowinset=.3,arrowlength=.5]{n}{m}\hspace*{6ex}
\end{math}
\begin{prooftree}\thickness=.08em
\[
  \[
    \justifies
    \dual a,a
  \]
  \justifies
  b\plus \dual a,\,a
  \using\plus_2
\]
\hspace*{1ex}
\[
  \[
    \justifies
    \dual a\,a
  \]
  \justifies
  b\plus \dual a,\,a
  \using\plus_2
\]
\justifies
b\plus \dual a\,,\,a\with a
\using\with
\end{prooftree}\hspace*{-3ex}\vspace*{8ex}
\end{center}

\begin{center}\hspace*{-8ex}%
\newcommand{\axgap}{\hspace{0ex}}
\begin{prooftree}\thickness=.08em
\[
  \[
    \justifies
    \leafOne,\leafThree
  \]
  \axgap
  \[
    \justifies
    \leafTwo,\leafThree
  \]
  \justifies
  \leafOne\withOne \leafTwo, \leafThree
  \using \withOne\!\!
\]
\hspace{0ex}
\[
  \[
    \justifies
    \leafOne,\leafFour
  \]
  \axgap
  \[
    \justifies
    \leafTwo,\leafFour
  \]
  \justifies
  \leafOne\withOne \leafTwo, \leafFour
  \using \withOne
\]
\justifies
\leafOne\withOne \leafTwo, \leafThree\withTwo \leafFour
\using \withTwo
\end{prooftree}
\begin{math}%
\hspace*{4ex}\Rnode{m}{\strut}\hspace*{3ex}\Rnode{n}{\strut}\ncline[arrows=->,linewidth=2pt,arrowsize=9pt,arrowinset=.3,arrowlength=.5]{m}{n}\hspace*{5ex}
{\Rnode{a1}{\atomOne}
\withOne
\Rnode{a2}{\atomOne}
\hspace{4ex}
\Rnode{a1'}{\dual \atomOne}
\withTwo
\Rnode{a2'}{\dual \atomOne}
\ncbar[angle=90,nodesep=2pt,arm=.65cm]{a1}{a2'}
\ncput[npos=1.5]{\pnode{12}}
\ncbar[angle=90,nodesep=2pt,arm=.2cm]{a2}{a1'}
\ncput[npos=1.5]{\pnode{21}}
\ncbar[angle=-90,nodesep=2pt,arm=.2cm]{a1}{a1'}
\ncput[npos=1.62]{\pnode{11}}
\ncbar[angle=-90,nodesep=2pt,arm=.65cm]{a2}{a2'}
\ncput[npos=1.4]{\pnode{22}}
\conflictstyle
\ncline{12}{21}
\ncline{11}{22}}
\hspace*{5ex}\Rnode{m}{\strut}\hspace*{3ex}\Rnode{n}{\strut}\ncline[arrows=->,linewidth=2pt,arrowsize=9pt,arrowinset=.3,arrowlength=.5]{n}{m}\hspace*{4ex}
\end{math}
\begin{prooftree}\thickness=.08em
\[
  \[
    \justifies
    \leafOne,\leafThree
  \]
  \axgap
  \[
    \justifies
    \leafOne,\leafFour
  \]
  \justifies
  \leafOne, \leafThree\withTwo \leafFour
  \using \withTwo\!\!
\]
\hspace{0ex}
\[
  \[
    \justifies
    \leafTwo,\leafThree
  \]
  \axgap
  \[
    \justifies
    \leafTwo,\leafFour
  \]
  \justifies
  \leafTwo, \leafThree\withTwo \leafFour
  \using \withTwo
\]
\justifies
\leafOne\withOne \leafTwo, \leafThree\withTwo \leafFour
\using \withOne
\end{prooftree}\hspace*{-8ex}
\end{center}

\begin{center}\vspace*{6ex}\newcommand{\branch}{\[
    \[
      \justifies
      c,\dual c
    \]
    \hspace*{0ex}
    \[
      \justifies
      a,\dual a
    \]
    \justifies
    a,c,\dual c\tensor\dual a
    \using\!\!\tensor\!\!\!
  \]}%
\hspace*{-5ex}\begin{prooftree}\thickness=.08em\proofdotseparation=4pt
\[
  \branch
  \hspace*{0ex}
  \branch
  \justifies
  a\,,\, c,\,(\dual c\tensor\dual a)\with (\dual c\tensor\dual a)
  \using\!\!\with\!\!
\]
\justifies a\parr c,\,(\dual c\tensor\dual a)\with (\dual c\tensor\dual a)
\using\!\parr\!\!
\end{prooftree}
\begin{math}%
\hspace*{3ex}\Rnode{m}{\strut}\hspace*{3ex}\Rnode{n}{\strut}\ncline[arrows=->,linewidth=2pt,arrowsize=9pt,arrowinset=.3,arrowlength=.5]{m}{n}\hspace*{5ex}
\Rnode{a}{a\strut}\parr \Rnode{c}{c\strut}
\hspace*{4ex}
(\Rnode{cc}{\strut\dual c}\tensor\Rnode{aa}{\strut\dual a})
\with 
(\Rnode{cc'}{\strut\dual c}\tensor\Rnode{aa'}{\strut\dual a})
\ncbar[angle=90,nodesep=2pt,arm=.6cm]{a}{aa}\ncput{\pnode{xa}}
\ncbar[angle=-90,nodesep=2pt,arm=.6cm]{a}{aa'}\ncput[npos=1.3]{\pnode{ya}}
\ncbar[angle=90,nodesep=2pt,arm=.2cm]{c}{cc}\ncput[npos=1.3]{\pnode{xc}}
\ncbar[angle=-90,nodesep=2pt,arm=.2cm]{c}{cc'}\ncput[npos=1.33]{\pnode{yc}}
\conflictline{xa}{xc}
\conflictline{ya}{yc}
\conflictline{xa}{yc}
\conflictline{xc}{ya}
\hspace*{5ex}\Rnode{m}{\strut}\hspace*{3ex}\Rnode{n}{\strut}\ncline[arrows=->,linewidth=2pt,arrowsize=9pt,arrowinset=.3,arrowlength=.5]{n}{m}\hspace*{3ex}
\end{math}
\begin{prooftree}\thickness=.08em\proofdotseparation=4pt
\[
  \branch
  \justifies a\parr c,\,\dual c\tensor\dual a
  \using\!\parr\!\!
\]
\hspace*{0ex}
\[
  \branch
  \justifies a\parr c,\,\dual c\tensor\dual a
  \using\!\parr\!\!
\]
\justifies a\parr c\,,\,(\dual c\tensor\dual a)\with (\dual c\tensor\dual a)
\using\!\!\with\!\!
\end{prooftree}\hspace*{-5ex}
\end{center}

\begin{center}\vspace*{5ex}\hspace*{1ex}%
\begin{prooftree}\thickness=.08em
\[
  \justifies
  \atomOne,\dual \atomOne
\]
\hspace{3ex}
\[
 \[
   \justifies
   \dual \atomTwo,\atomTwo
 \]
 \hspace{2ex}
 \[
   \justifies
   \dual \atomTwo,\atomTwo
 \]
 \justifies
 \dual \atomTwo,\atomTwo\with \atomTwo
 \using \with
\]
\justifies
\atomOne, \dual \atomOne\tensor \dual \atomTwo, \atomTwo\with \atomTwo
\using \tensor
\end{prooftree}
\hspace*{24ex}
\newcommand{\slice}{%
\[
  \[
    \justifies
    \atomOne,\dual \atomOne
  \]
  \hspace{2ex}
  \[
    \justifies
    \dual \atomTwo,\atomTwo
  \]
  \justifies
  \atomOne,\dual \atomOne\tensor \dual \atomTwo,\atomTwo
  \using \tensor
\]}
\begin{prooftree}\thickness=.08em
\slice\hspace{2ex}\slice
\justifies
\atomOne, \dual \atomOne\tensor \dual \atomTwo, \atomTwo\with \atomTwo
\using \with
\end{prooftree}\hspace*{-4ex}
\vspace{1ex}\end{center}

\begin{center}
\hspace*{-3ex}\pnode{m}\hspace*{50ex}\pnode{m'}

\vspace*{.5ex}

\hspace*{-3ex}\pnode{n}\hspace*{50ex}\pnode{n'}%
\ncline[arrows=->,linewidth=2pt,arrowsize=9pt,arrowinset=.3,arrowlength=.5]{m}{n}%
\ncline[arrows=->,linewidth=2pt,arrowsize=9pt,arrowinset=.3,arrowlength=.5]{m'}{n'}
\end{center}

\begin{center}%
\vspace*{3ex}%
\newcommand{\gap}{\hspace{5ex}}%
\newcommand{\sequent}{\begin{math}%
\Rnode{a}{a}
\gap
\Rnode{a'}{\dual a}
\tensor
\Rnode{\atomTwo'}{\dual \atomTwo}
\gap
\Rnode{\atomTwo1}{\atomTwo}
\with
\Rnode{\atomTwo2}{\atomTwo}
\end{math}}%
\sequent%
\ncbar[angle=90,nodesep=2pt,arm=.2cm]{a}{a'}
\ncbar[angle=-90,nodesep=2pt,arm=.2cm]{\atomTwo'}{\atomTwo1}
\ncput{\pnode{\atomTwo'\atomTwo1}}
\ncbar[angle=90,nodesep=2pt,arm=.2cm]{\atomTwo'}{\atomTwo2}
\ncput{\pnode{\atomTwo'\atomTwo2}}
\hspace*{30ex}
\sequent%
\ncbar[angle=90,nodesep=2pt,arm=.2cm]{a}{a'}
\ncput{\pnode{aa'2}}
\ncbar[angle=-90,nodesep=2pt,arm=.2cm]{a}{a'}
\ncput{\pnode{aa'1}}
\ncbar[angle=90,nodesep=2pt,arm=.2cm]{\atomTwo'}{\atomTwo2}
\ncput{\pnode{\atomTwo'\atomTwo2}}
\ncbar[angle=-90,nodesep=2pt,arm=.2cm]{\atomTwo'}{\atomTwo1}
\ncput{\pnode{\atomTwo'\atomTwo1}}
{\conflictstyle
\nccurve[angleA=-40,angleB=-140]{aa'1}{\atomTwo'\atomTwo1}
\nccurve[angleA=40,angleB=140]{aa'2}{\atomTwo'\atomTwo2}}
\vspace*{3ex}
\end{center}

\caption{\label{fig-egs}Illustrating the surjective translation function from proofs to conflict nets.
The first three rows show how conflict nets are invariant with
respect to raising a $\oplus$-, $\with$- or
\usebox{\parrbox}-rule over a $\with$-rule, respectively: each pair of
proofs (left and right) maps to the same conflict net (centre).
The last two translations (flowing downwards) show that
raising a $\otimes$-rule over a $\with$-rule changes the conflict net;
this seems to be the price of p-time proof translation.
Conflicts between axiom links are shown as dotted edges.  Axiom links
which overlap (share an atom in the sequent) conflict
implicitly.}\vspace*{-2.4ex}
\end{figure}
Examples of conflict nets are shown in Figure~\ref{fig-egs}.
Figure~\ref{fig-egs} also illustrates how translation from a proof to a
conflict net is invariant with respect to raising a
$\plus\!/\!\with\!/\!\parr$-rule over a $\with$-rule.
Table~\ref{table-summary} compares different proof nets.

\begin{table}
\newcommand{\monotransnote}{a}
\newcommand{\brokennote}{b}
\newcommand{\pricenote}{c}
\newcommand{\pelimnote}{d}
\newcommand{\nameparboxwidth}{10ex}
\newcommand{\vthick}{|@{\hspace{.1pt}}|@{\hspace{.1pt}}|}
\begin{center}
\begin{tabular}{@{\hspace{0ex}}l\vthick c|c\vthick c|c\vthick c|c|}
\cline{2-7}
&\multicolumn{1}{c}{}&&\multicolumn{1}{c}{}&&\multicolumn{1}{c}{}&
\\[-.5ex]
&
\multicolumn{1}{r}{\begin{picture}(0,0)\put(0,0){\makebox(11,0)[b]{\raisebox{0ex}[0ex][0ex]{\sl Representation efficiency}}}\end{picture}} 
&&
\multicolumn{1}{r}{\begin{picture}(0,0)\put(-1,0){\makebox(0,0)[b]{\sl Abstraction}}\end{picture}} 
&&
\multicolumn{1}{r}{\begin{picture}(0,0)\put(0,0){\makebox(11,0)[b]{\raisebox{0ex}[0ex][0ex]{\sl Cut elimination}}}\end{picture}}
&
\\[2ex]\hline\vline&&&&&&\\[-1.5ex]
\vline\;\;\sl Proof net 
&
\!\parbox{11ex}{\raggedright\bf P-time correctness}\!
&
\!\parbox{10.5ex}{\raggedright\bf P-time translation\!}\!
&
\!\parbox{12ex}{\raggedright\bf Raise \mbox{$\parr\!/\!\plus\!/\!\with$-rule} over \mbox{$\with$-rule}\!\!}\!\!
&
\!\parbox{12.5ex}{\raggedright\bf Raise \mbox{$\tensor$-rule} over \mbox{$\with$-rule}}\!\!
&
\parbox[c]{7.5ex}{\;\;\;\bf P-time\\[-.3ex]}
&
\!\parbox[c]{10ex}{\bf Confluent\\\small(unit-free)}\!\!
\\[2.8ex]\Hline\vline&&&&&&\\[-1.5ex]
\vline\;\;\parbox{\nameparboxwidth}{\raggedright Box
\\\cite{Gir87}}\!
&
\chmark
&
\chmark
&
\cross
&
\cross
&
\probno
&
\openq
\\[2ex]\hline\vline&&&&&&\\[-1.5ex]
\vline\;\;\parbox{\nameparboxwidth}{\raggedright Monomial \cite{Gir96}}\!
&
\openq
&
\chmark\ftnotemarknospace{\monotransnote}
&
\cross\ftnotemarknospace{\monotransnote}
&
\cross\ftnotemarknospace{\monotransnote}
&
\probno
&
\openq\ftnotemarknospace{\brokennote}
\\[2ex]\hline\vline&&&&&&\\[-1.5ex]
\vline\;\;\parbox{\nameparboxwidth}{\raggedright Slice [HG03,05]}\!
&
\chmark
&
\cross
&
\chmark
&
\chmark
&
\chmark\ftnotemarknospace{\pelimnote}
&
\chmark
\\[2ex]\hline\vline&&&&&&\\[-1ex]
\vline\;\;\parbox{\nameparboxwidth}{\raggedright Conflict}\!
&
\chmark
&
\chmark
&
\chmark
&
\cross\ftnotemarknospace{\pricenote}
&
\probno
&
\openq
\\[1.8ex]\hline
\end{tabular}\!

\vspace*{1ex}
\end{center}

\hspace*{10ex}
\chmark=yes
\hspace{3ex}
\cross=no
\hspace{3ex}
\openq=open question
\hspace{3ex}
\probno=open question, probably no

\vspace*{1ex}

\hspace*{2ex}\ftnotetextmark{\monotransnote}With respect to the canonical (non-surjective) proof translation \cite[p.\,7]{Gir96}.
See footnote~\ref{note-monomial-box}.

\hspace*{2ex}\ftnotetextmark{\brokennote}The definition proposed in \cite[p.\,24]{Gir96} does not work: see Section~\ref{monomial-cut-elim}.

\hspace*{2ex}\ftnotetextmark{\pricenote}Seemingly the price of having a p-time translation from proofs.

\hspace*{2ex}\ftnotetextmark{\pelimnote}P-time since normalisation is slicewise.

\vspace*{2ex}
\caption{\label{table-summary}Comparison of proof nets.}\vspace*{1ex}\hrule
\end{table}

\paragraph*{Related work. }  
The last few years have seen a renaissance of work involving MALL
proof nets, including \cite{Ham04} (extending monomial nets with mix,
analysing softness), \cite{CP05} (a language for MALL proofs, viewed
as processes), \cite{CF05} (a ludics-based analysis of
sequentiality/parallelism), \cite{BHS05} (a fully complete
\emph{relational} model for MALL), \cite{Mai07} (extending Danos 
contractibility \cite{Dan90} to additives, using a distributivity
rewrite), \cite{Abr07} (a domain-theoretic view of unfolding the
$\with$-rules as we go up a proof), to name but a few.\footnote{With
polarization, proof nets become much easier: see
\cite{LT04}.}

In each case the underlying data structure involved are more complex
than a conflict net, carrying additional machinery such as monomial
weights on subformulas, subformula occurrences, focalisation,
contraction nodes, domains, partial left/right resolutions of the
$\with$'s in a sequent, and so on.  Like box nets and monomial nets,
most deal with occurrences of
\emph{subformulas}; the data structure of a conflict net involves only atoms,
true to the spirit of the geometry of interaction \cite{Gir89}.  By
not dealing with internal nodes of subformula trees, which are
sequential, conflict nets are in some sense maximally parallel.

Current work for conflict nets includes arranging them into a
category, possibly via a strongly normalising cut elimination.  A
naive cut elimination can be obtained by emulating the elimination of
box nets (copying empires around).
One possible approach is to try and use pullbacks of (contractible)
coherence spaces to obtain a completely abstract form of cut
hyper-elimination (composition) in a compact closed category.  If it
worked out, this would ensure a forgetful functor to the underlying
compact closed composition of slice nets.

Conflict nets are a variant of (and were inspired by)
\emph{combinatorial proofs} introduced recently for classical logic \cite{Hug06,Hug06i}:
each conflict net can be viewed as a maximal map (homomorphism) of
contractible coherence spaces ($P_4$-free graphs, or cographs), from
axioms to sequent.  The relationship with combinatorial proofs is
sketched in Section~\ref{sec-comb}.

\paragraph*{Acknowledgement. }  I'm grateful to Robin Houston for 
discussions about abstract categorical versions of cut elimination,
based on pullbacks of coherence spaces.  In particular, Robin showed
me how to construct pullbacks in the category of coherence spaces.
I'm also indebted to Roberto Maieli, whose extension of Danos'
contractability criterion \cite{Mai07} stimulated me to think about
MALL proof nets again.

\section{Preliminaries}

\subsection{MALL}\label{sec-mall}

We work with \emph{cut-free, unit-free multiplicative-additive linear
logic} \cite{Gir87}, henceforth denoted MALL.  

Fix a set $\mathcal{A}=\{a,b,c,\ldots\}$ of
\defn{literals} equipped with a function
$\overline{(\rule{1ex}{0ex}\rule{0ex}{1.3ex})}:\mathcal{A}\to\mathcal{A}$
such that $\dual a\neq a$ and $\dual{\dual a}=a$ for all
$a\in\mathcal{A}$.
MALL formulas are generated from literals by the binary connectives
$\tensor$ (tensor) $\parr$ (par) $\with$ (with) and $\plus$ (plus).
Define $\dual\tensor = \parr$, $\dual\parr=\tensor$,
$\dual\plus=\with$ and $\dual\with=\plus$.
Define negation $(.)\perpp$ by $a\perpp=\dual a$ on literals, and
$(A\square B)\perpp=A\perpp\dual\square B\perpp$.
Formulas $A$ and $A\perpp$ are \defn{dual}.
A sequent is a list (finite sequence) $A_1,\ldots,A_n$ of formulas
($n\ge 0$).
Throughout this document we take $P,Q,R,\ldots$ to range over
literals, $A,B,C,\ldots$ over formulas, and
$\Gamma,\Delta,\Sigma,\ldots$ over sequents.

We identify a formula with its parse tree: a tree
with leaves labelled with literals and internal vertices labelled with
connectives, equipped with a linear order on leaves.  Edges are
oriented away from the leaves.
We identify a sequent with its parse forest: the disjoint union of its
formulas (formula parse trees), with a linear order on leaves.  For
example, we identify the three-formula sequent
$\hspace{.5ex}\atomOne\,,\,\dual\atomOne\tensor\dual\atomTwo\,,\,(\atomTwo\with\atomTwo)\tensor\atomOne\hspace{.5ex}$
with the following parse forest:
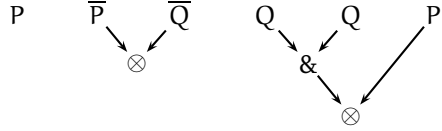
\begin{center}
\begin{math}\psset{nodesep=1pt}
\newcommand{\gap}{\hspace{5ex}}
\Rnode{1}{\atomOne}
\gap
\Rnode{2}{\dual\atomOne}
\gap
\Rnode{3}{\dual\atomTwo} 
\gap
\Rnode{4}{\atomTwo}
\gap
\Rnode{5}{\atomTwo}
\gap
\Rnode{6}{\atomOne}
\ncline[linestyle=none]{2}{3}
\nbput[labelsep=.5cm]{\rnode{t}{\tensor}}
\ncline{->}{2}{t}
\ncline{->}{3}{t}
\ncline[linestyle=none]{4}{5}
\nbput[labelsep=.5cm]{\rnode{w}{\with}}
\ncline{->}{4}{w}
\ncline{->}{5}{w}
\ncline[linestyle=none]{4}{6}
\nbput[labelsep=1.2cm]{\rnode{p}{\tensor}}
\ncline{->}{w}{p}
\ncline{->}{6}{p}
\end{math}
\vspace*{8ex}
\end{center}
The linear order on leaves is given by the left-to-right order on the page.
Two leaves are \defn{dual} if their literal labels are dual.

If $\Gamma=A_1,\ldots,A_n$, and $\sigma$ be a permutation on $n$
(\ie, a bijection $\{1,\ldots,n\}\to\{1,\ldots,n\}$),
write $\sigma\Gamma$ for the sequent $A_{\sigma 1},\ldots,A_{\sigma n}$.
Proofs are generated using the
rules in Figure~\ref{fig-rules}.
\begin{figure}%
\begin{center}
\begin{prooftree}\thickness=.08em
\strut
\justifies\;
P,\dual P
\;\using\axlabel
\end{prooftree}
\hspace*{10ex}
\begin{prooftree}\thickness=.08em
\Gamma
\justifies\;
\sigma \Gamma
\;\using\permlabel_\sigma\hspace*{-3ex}
\end{prooftree}

\vspace{5ex}

\begin{prooftree}\thickness=.08em
\Gamma,\,A\,,\,B
\justifies\;
\Gamma,\,A\parr B
\;\using\parr
\end{prooftree}
\hspace*{10ex}
\begin{prooftree}\thickness=.08em
\Gamma,\;\;A_i\;\;
\justifies\;
\Gamma,\,A_o\plus A_1
\;\using\plus_i
\end{prooftree}

\vspace{5ex}

\begin{prooftree}\thickness=.08em
\Gamma,A
\hspace*{3ex}
B,\Delta
\justifies\;
\Gamma,A\tensor B,\Delta
\;\using\tensor
\end{prooftree}
\hspace*{12ex}
\begin{prooftree}\thickness=.08em
\Gamma,A
\hspace*{3ex}
\Gamma,B
\justifies\;
\Gamma,A\with B
\;\using\with
\end{prooftree}
\end{center}
\caption{\label{fig-rules}MALL proof rules.  Here $\sigma$ is any permutation on $n$, 
the number of formulas in the sequent $\Gamma$ above the $\permlabel$-rule.}\vspace*{1ex}\hrule
\end{figure}
As a technical convenience, we shall often supress permutation
($\permlabel$) rules, for example, writing
\begin{center}
\begin{prooftree}\thickness=.08em
\Gamma,\,A,\,B,\Delta
\justifies\;
\Gamma,\,A\parr B,\Delta
\;\using\parr
\end{prooftree}
\end{center}
which leaves implicit a permutation rule above and below the
$\parr$-rule, if $\Delta$ is non-empty.

\subsection{Coherence spaces}

We write $\adjacent$ for strict coherence and $\conflict$ for strict
incoherence of coherence spaces \cite[\S3]{Gir87}.  We call
$\adjacent$
\defn{adjacency} and $\conflict$ \defn{conflict}.  
The elements of the web $|X|$ of a coherence space $X$ are
\defn{tokens} of $X$.
Recall that a \defn{map} $X\to Y$ between coherence spaces is a binary
relation $R\subseteq |X|\times|Y|$ which preserves strict coherence and
reflects strict incoherence: $y_1R^{\text{op}}x_1\girstcoh x_2Ry_2$ implies
$y_1\girstcoh y_2$, and $x_1Ry_1\conflict y_2R^{\text{op}}x_2$ implies
$x_1\conflict x_2$.  (We write $xRy$ or $yR^{\text{op}}x$ for $\langle x,y\rangle\in R$.)

\section{Conflict linkings}

\paragraph*{Informal definition.}
A \emph{link} on a sequent $\Gamma$ is an edge between dual leaves.  A
\emph{linking}
on $\Gamma$ is a finite set $L$ of links on $\Gamma$ equipped with a
symmetric, irreflexive binary
\emph{conflict} relation
$\conflict\,\subseteq\,L\times L$ such that overlap implies conflict:
if distinct links $l$ and $m$ share an atom, then $l\conflict m$.
Links may be parallel (between the same pair of leaves).  Examples of
linkings are shown in Figure~\ref{fig-egs}.  When drawing linkings, we
leave implicit the conflicts implied by overlap.

\paragraph{Formalisation.}
A \defn{dual pair} in $\Gamma$ is a pair $\{x,y\}$ of dual leaves in
$\Gamma$.
\begin{definition}
A \defn{linking} on $\Gamma$ is a binary relation $\lambda :
L\to\leavesof{\Gamma}$ from a finite coherence space $L$, whose tokens
are called \defn{links} on $\Gamma$,
to the set $\leavesof{\Gamma}$ of leaves in $\Gamma$, satisfying:
\begin{itemize}
\item
\defn{Dual pair.}
For every link $l$ in $L$ the direct image
$\lambda[l]=\{x\in\leavesof{\Gamma}:\langle l,x\rangle\in\lambda\}$
is a dual pair.
\item
\defn{Overlap.}
If $\langle l,x\rangle\in\lambda$ and $\langle
l',x\rangle\in\lambda$ with $l\neq l'$
($l$ and $l'$ \defn{overlap} at $x$) 
then $l\conflict l'\cohmod{L}$.
\end{itemize}
\end{definition}
We abbreviate a linking $\lambda : L\to\leavesof{\Gamma}$ to $\lambda :
L\to\Gamma$ or \raisebox{0ex}{$L\map{\lambda}{5ex}\Gamma$}.

\section{P-time proof translation function from proofs}

\paragraph*{Informal definition.}
A MALL proof of $\Gamma$ translates to a linking on $\Gamma$ by
viewing each axiom rule as a link on $\Gamma$ (by tracing its two leaves
down the proof into $\Gamma$), and defining $l\conflict m$
iff $l$ and $m$ are in opposite branches above a $\with$-rule.
Figure~\ref{fig-egs} shows examples of proof translation.

\paragraph*{Formalisation.}
The following formalisation is by induction on the number of rules in
a proof.
\begin{itemize}
\item
\emph{Base case. }
The axiom rule
\raisebox{.5ex}[0ex]{\;\;$\begin{prooftree}\thickness=.08em
\justifies \raisebox{.3ex}{$\,P,\,\dual P\,$}
\end{prooftree}$\;\;}
translates to the unique single-link linking on $P,\dual P$.\footnote{If $x$
and $x'$ are the two leaves, the linking is $\,\lambda : I\to P,\dual P\,$
where $I$ has a single token $\bullet$ and
$\lambda=\{\langle\bullet,x\rangle,\langle\bullet,x'\rangle\}$.}
\item
\emph{Inductive step. }
Every instance of a rule induces an inclusion function from the leaves
of each sequent above the line to the sequent below the
line.\footnote{Each sequent (parse forest) above the line is a subgraph of the
sequent below the line.}
Via these leaf inclusions, Figure~\ref{fig-translation}
interprets each rule as an operation on linkings.
\begin{figure}%
\begin{center}
\begin{prooftree}\thickness=.08em
L\map{\lambda}{5ex}\Gamma
\justifies\;\rule{0ex}{3ex}
L\umap{\lambda}{5ex}\sigma\Gamma
\;\using\permlabel_\sigma
\end{prooftree}
\hspace*{5ex}
\begin{prooftree}\thickness=.08em
L\map{\lambda}{5ex}\Gamma,\,A\,,\,B
\justifies\;\rule{0ex}{3ex}
L\umap{\lambda}{5ex}\Gamma,\,A\parr B
\;\using\parr
\end{prooftree}
\hspace*{8ex}
\begin{prooftree}\thickness=.08em
L\map{\lambda}{5ex}\Gamma,\;\;A_i\;\;
\justifies\;\rule{0ex}{3ex}
L\umap{\lambda}{5ex}\Gamma,\,A_1\plus A_2
\;\using\plus_i
\end{prooftree}

\vspace{8ex}

\begin{prooftree}\thickness=.08em
L\map{\lambda}{5ex}\Gamma,A
\hspace*{8ex}
M\map{\mu}{5ex}B,\Delta
\justifies\;\rule{0ex}{3.2ex}
L\times M\umap{\lambda\cup\mu}{8ex}\Gamma,A\tensor B,\Delta
\;\using\tensor
\end{prooftree}
\hspace*{12ex}
\begin{prooftree}\thickness=.08em
L\map{\lambda}{5ex}\Gamma,A
\hspace*{8ex}
M\map{\mu}{5ex}\Gamma,B
\justifies\;\rule{0ex}{3.2ex}
L+ M\umap{\lambda\cup\mu}{8ex}\Gamma,A\with B
\;\using\with
\end{prooftree}
\end{center}
\caption{\label{fig-translation}Inductive translation from a proof to a conflict linking.}\vspace*{1ex}\hrule
\end{figure}
The sum $L+M$ in the interpretation of the $\with$-rule is the
disjoint union (categorical sum/coproduct) of the coherence spaces $L$
and $M$, denoted $L\plus M$ in \cite{Gir87}.
Without loss of generality, we assume the canonical injections from
the token sets of $L$ and $M$ into the token set of $L+M$ are
inclusions.
The product $L\times M$ in the interpretation of the $\tensor$-rule is
$L+M$ together with strict coherence between every token in
$L$ and every token in $M$.  This is categorical product, denoted
$L\with M$ in \cite{Gir87}.
\end{itemize}
Each rule interpretation preserves the \emph{Dual pair} and \emph{Overlap}
conditions in the definition of a linking.  Thus the translation of a
proof is a well-defined linking.

A linking is \defn{sequentializable} if it is the translation of a
proof; any such a proof is a \defn{sequentialization} of the linking.

\section{Slicings}

This section defines a \emph{slicing} as a refinement of a linking, a
stepping stone towards the definition of conflict net.

A coherence space is \defn{contractible} if its web is
finite and $P_4$-free (no induced four-vertex path \cite{Sei74}):
whenever $x_1\conflict x_2\conflict x_3\conflict x_4$ for distinct
$x_i$ then $x_1\conflict x_3$ or $x_2\conflict x_4$ or $x_1\conflict
x_4$
\cite{Hu99}.
Define $\cohof{\Gamma}$ as the coherence space whose
tokens are the leaves of $\Gamma$ with
$x\conflict y$ iff $x\neq y$ and the smallest subformula containing
$x$ and $y$ is additive.\footnote{In other words, $x\conflict y$ iff
$x$ and $y$ are in the same formula $A$, and the first
common vertex along the paths from $x$ and $y$ to the root of $A$ is labelled $\with$ or $\plus$.
Equivalently, the
join (least upper bound) $z$ of $x$ and $y$ exists when we interpret
$\Gamma$ as a partial order with leaves maximal and roots minimal, and
$z$ is labelled $\with$ or $\plus$.}  
If $\Gamma$ is non-empty, its coherence space $\cohof{\Gamma}$ is
contractible (a simple induction).
\begin{definition}
A \defn{slicing} on $\Gamma$ is a maximal map $\lambda :
L\to\cohof\Gamma$ from a contractible coherence space $L$.
\end{definition}
Maximality is with respect to inclusion among maps
$L\to\cohof{\Gamma}$.\footnote{Thus $\lambda$ is maximal iff it is a
maximal clique in $L\lolly\cohof{\Gamma}$.}
An example of a slicing
is shown in
Figure~\ref{fig-slicing} with its underlying maximal map clarified.
\begin{figure}
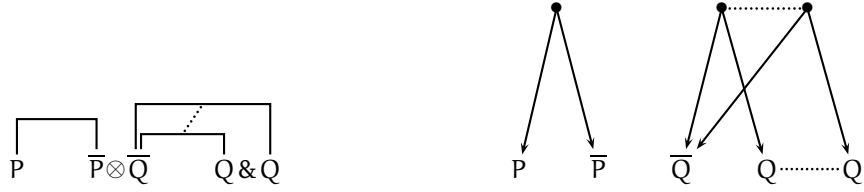

\begin{center}%
\vspace*{13ex}%
\begin{math}%
\newcommand{\gap}{\hspace{5ex}}%
\newcommand{\height}{12ex}
\Rnode{P}{P}
\gap
\Rnode{P'}{\dual P}
\tensor
\Rnode{Q'}{\dual Q}
\gap
\Rnode{Q1}{Q}
\with
\Rnode{Q2}{Q}
\ncbar[angle=90,nodesep=2pt,arm=.4cm]{P}{P'}
\ncbar[angle=90,nodesep=2pt,arm=.2cm,offsetA=-1pt]{Q'}{Q1}
\ncput{\pnode{Q'Q1}}
\ncbar[angle=90,nodesep=2pt,arm=.6cm,offsetA=1pt]{Q'}{Q2}
\ncput{\pnode{Q'Q2}}
\conflictline{Q'Q1}{Q'Q2}
\hspace*{18ex}
\psset{nodesepA=.5pt,nodesepB=.5pt}%
\Rnode{P}{P}
\gap
\Rnode{P'}{\dual P}
\gap
\Rnode{Q'}{\dual Q}
\gap
\Rnode{Q1}{Q}
\gap
\Rnode{Q2}{Q}
\conflictline{Q1}{Q2}
\psset{nodesepB=2pt}
\ncline[linestyle=none]{P}{P'}
\naput[labelsep=\height]{\rnode{l}{\bullet}}
\ncline[nodesepA=-1pt]{->}{l}{P}
\ncline[nodesepA=-1pt]{->}{l}{P'}
\ncline[linestyle=none]{Q'}{Q1}
\naput[labelsep=\height]{\rnode{m}{\bullet}}
\ncline[nodesepA=-1pt]{->}{m}{Q'}
\ncline[nodesepA=-1pt]{->}{m}{Q1}
\ncline[linestyle=none]{Q1}{Q2}
\naput[labelsep=\height]{\rnode{n}{\bullet}}
\ncline[nodesepA=-1pt]{->}{n}{Q'}
\ncline[nodesepA=-1pt]{->}{n}{Q2}
\conflictline[nodesep=0pt]{m}{n}
\vspace*{1ex}
\end{math}
\end{center}
\caption{\label{fig-slicing}An example of a slicing $\lambda : L\to \Gamma$ with $\Gamma=P,\dual P\otimes\dual Q,Q,Q$.  
The underlying maximal map $\lambda : L\to\cohof{\Gamma}$ between
contractible coherence spaces is shown on the right.  The fives tokens of
the coherence space $\cohof\Gamma$ are shown with their literal
labels.  The three links (tokens) of $L$ are shown as
$\bullet$.
Note that $\lambda$ is indeed maximal: were we to add any edge to the
binary relation $\lambda$, it would no longer be a coherence space
map.}\vspace{1ex}\hrule
\end{figure}

\begin{proposition}
Checking that a linking $\lambda : L\to\Gamma$ is a slicing is
p-time in the sizes of\/ $L$ and\/ $\Gamma$.
\end{proposition}
\begin{proof}
Checking that $\lambda$ is a map (preserving $\girstcoh$ and
reflecting $\conflict$) is clearly polynomial.  Checking
contractibility ($P_4$-freeness) is linear \cite{CPS85}.
Checking direct images are dual pairs is obviously polynomial.
Checking maximality is polynomial: for every edge $e\not\in\lambda$ we
check that $\lambda\cup\{e\}$ is not a map.\footnote{It suffices to
test with single extra edges $e$ since a map $R:X\to Y$ is maximal
iff it is a maximal clique in the coherence space $X\lolly Y$.}
\end{proof}
A \defn{slice} of a slicing $\lambda : L\to\cohof{\Gamma}$ is a
maximal clique in $L$.\footnote{A clique $C$ is a set of pairwise
coherent tokens: if $x,y\in C$ and $x\neq y$ then $x\girstcoh y$.}
The two slices of the example in Figure~\ref{fig-slicing} are illustrated below.
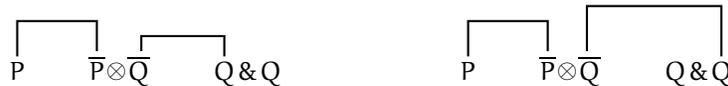
\begin{center}%
\vspace*{3ex}%
\begin{math}%
\newcommand{\gap}{\hspace{5ex}}%
\newcommand{\height}{12ex}
\Rnode{P}{P}
\gap
\Rnode{P'}{\dual P}
\tensor
\Rnode{Q'}{\dual Q}
\gap
\Rnode{Q1}{Q}
\with
\Rnode{Q2}{Q}
\ncbar[angle=90,nodesep=2pt,arm=.4cm]{P}{P'}
\ncbar[angle=90,nodesep=2pt,arm=.2cm,offsetA=-1pt]{Q'}{Q1}
\hspace*{14ex}
\Rnode{P}{P}
\gap
\Rnode{P'}{\dual P}
\tensor
\Rnode{Q'}{\dual Q}
\gap
\Rnode{Q1}{Q}
\with
\Rnode{Q2}{Q}
\ncbar[angle=90,nodesep=2pt,arm=.4cm]{P}{P'}
\ncbar[angle=90,nodesep=2pt,arm=.6cm,offsetA=1pt]{Q'}{Q2}
\end{math}
\end{center}
An \defn{additive resolution} of $\Gamma$ is a maximal clique in
$\cohof{\Gamma}$ 
\cite{HG03,HG05}.
The \defn{image} of a set $Z\subseteq X$ under a binary relation $R\subseteq X\times Y$ 
is $\{\,y\in Y:zRy\text{ for some }z\in Z\,\}$.
The following proposition formalises the sense in which ``every slice
is an MLL linking'' (\cf\ \cite{Gir87,Gir96,HG03,HG05}).
\begin{proposition}\label{prop-slices}
\!\!Let\/ $\lambda : L\to\Gamma$ be a non-empty slicing.\!
The image of every slice of\/ $\lambda$ is an additive
resolution of\/ $\Gamma$.
\end{proposition}
\begin{proof}
A corollary of \cite[Prop.\,2.2]{Hu99}: a non-empty map between contractible
coherence spaces is maximal iff it preserves maximal cliques, \ie, the
image of any maximal clique is a maximal clique.
\end{proof}
Note that the proposition holds for the two slices depicted above.
The proposition is somewhat surprising, since checking every slice
appears exponential-time (because a slice is a subset).

\section{Introducing erasure: Boxless nets}\label{sec-implicit}

In Section~\ref{sec-erasure} we define a conflict net as a slicing
which is \emph{erasable} under a confluent, terminating (strongly
normalising) \emph{erasure} rewrite $\leadsto$.
Erasability is checkable in p-time in the number of links and in the
number of leaves in the sequent.
A form of erasure will also yield p-time correctness for the slice
nets of \cite{HG03,HG05}.
For didactic purposes, we begin by defining erasure in a simple
setting related to box nets \cite{Gir87}, since that is the most
likely to be familiar to the reader.  However, the reader can safely
skip to Section~\ref{sec-erasure} without loss of continuity.

We shall describe a variant of box nets in which the
circumscribing boxes are not drawn explicitly.  
Accordingly, we shall refer to them as \emph{boxless nets}.  
The translation from a proof to a boxless net is exactly the same as
the translation to a box net --- only one forgets to draw the boxes.
For example, the two proofs on page~\ref{intro-proofs} translate
(respectively) to the following pair of box nets:
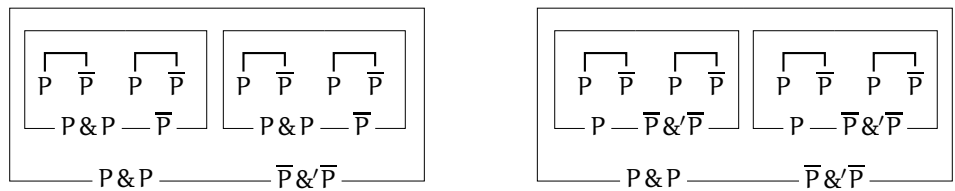
\begin{center}
\begin{math}
\begin{pspicture}[nodesep=2pt](-3,-.3)(3,2.5)
\newcommand{\innerbox}[3]{
  \rput(-.5,0){\rnode{bb}{\strut$#1\with #2$}}
  \rput(.5,0){\rnode{y}{\strut$#3$}}
  \pnode(0,1.3){lefttop}
  \ncline[linestyle=solid,linewidth=.3pt]{bb}{y}
  \ncbar[linestyle=solid,linewidth=.3pt,nodesepB=0pt,angle=-180,arm=.4cm]{bb}{lefttop}
  \ncbar[linestyle=solid,linewidth=.3pt,nodesepB=0pt,angle=0,arm=.4cm]{y}{lefttop}
  \rput(-.8,.6){$\rnode{a}{#1}\hspace*{2ex}\rnode{z}{#3}$}
  \ncbar[angle=90,arm=5pt]{a}{z}
  \rput(.4,.6){$\rnode{a}{#2}\hspace*{2ex}\rnode{z}{#3}$}
  \ncbar[angle=90,arm=5pt]{a}{z}
}
\rput(-1.2,0){\rnode{aa}{\strut$\leafOne\with\leafTwo$}}
\rput(1.2,0){\rnode{zz}{\strut$\leafThree\with\mkern-4mu'\leafFour$}}
\pnode(0,2.3){top}
\ncline[linestyle=solid,linewidth=.3pt]{aa}{zz}
\ncbar[linestyle=solid,linewidth=.3pt,nodesepB=0pt,angle=-180,arm=1.1cm]{aa}{top}
\ncbar[linestyle=solid,linewidth=.3pt,nodesepB=0pt,angle=0,arm=1.1cm]{zz}{top}
\rput(-1.2,.7){\innerbox{\leafOne}{\leafTwo}{\leafThree}}
\rput(1.44,.7){\innerbox{\leafOne}{\leafTwo}{\leafFour}}
\end{pspicture}
\hspace*{6ex}
\begin{pspicture}[nodesep=2pt](-3,-.3)(3,2.5)
\newcommand{\innerbox}{%
  \rput(-.5,0){\rnode{b}{\strut$\leafOne$}}
  \rput(.5,0){\rnode{yy}{\strut$\leafThree\with\mkern-4mu'\leafFour$}}
  \pnode(0,1.3){lefttop}
  \ncline[linestyle=solid,linewidth=.3pt]{b}{yy}
  \ncbar[linestyle=solid,linewidth=.3pt,nodesepB=0pt,angle=-180,arm=.4cm]{b}{lefttop}
  \ncbar[linestyle=solid,linewidth=.3pt,nodesepB=0pt,angle=0,arm=.4cm]{yy}{lefttop}
  \rput(-.4,.6){$\rnode{a}{\leafOne}\hspace*{2ex}\rnode{z}{\leafThree}$}
  \ncbar[angle=90,arm=5pt]{a}{z}
  \rput(.8,.6){$\rnode{a}{\leafTwo}\hspace*{2ex}\rnode{z}{\leafFour}$}
  \ncbar[angle=90,arm=5pt]{a}{z}
}
\rput(-1.2,0){\rnode{aa}{\strut$\leafOne\with\leafTwo$}}
\rput(1.2,0){\rnode{zz}{\strut$\leafThree\with\mkern-4mu'\leafFour$}}
\pnode(0,2.3){top}
\ncline[linestyle=solid,linewidth=.3pt]{aa}{zz}
\ncbar[linestyle=solid,linewidth=.3pt,nodesepB=0pt,angle=-180,arm=1.1cm]{aa}{top}
\ncbar[linestyle=solid,linewidth=.3pt,nodesepB=0pt,angle=0,arm=1.1cm]{zz}{top}
\rput(-1.44,.7){\innerbox}
\rput(1.2,.7){\innerbox}
\end{pspicture}
\end{math}
\end{center}
Now emphasise the superposition/contraction of these formulas, and
drop the surrounding boxes:
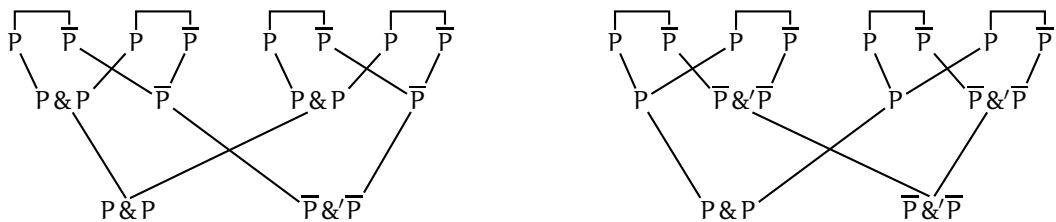
\begin{center}\vspace{2ex}
\begin{math}
\psset{nodesep=1pt}
\newcommand{\litgap}{\hspace*{3ex}}
\newcommand{\linkgap}{\hspace*{4ex}}
\newcommand{\midgap}{\hspace*{5ex}}
\newcommand{\botshift}{\hspace*{4ex}}
\begin{array}{c@{\hspace*{5ex}}c}
\Rnode{leafOneATop}{\leafOne}
\litgap
\Rnode{leafThreeATop}{\leafThree}
\linkgap
\Rnode{leafTwoBTop}{\leafTwo}
\litgap
\Rnode{leafThreeBTop}{\leafThree}
&
\Rnode{leafOneCTop}{\leafOne}
\litgap
\Rnode{leafFourCTop}{\leafFour}
\linkgap
\Rnode{leafTwoDTop}{\leafTwo}
\litgap
\Rnode{leafFourDTop}{\leafFour}
\ncbar[angle=90,arm=5pt]{leafOneATop}{leafThreeATop}
\ncbar[angle=90,arm=5pt]{leafTwoBTop}{leafThreeBTop}
\ncbar[angle=90,arm=5pt]{leafOneCTop}{leafFourCTop}
\ncbar[angle=90,arm=5pt]{leafTwoDTop}{leafFourDTop}
\\[2ex]
\Rnode{leafOneAMid}{\leafOne}
\Rnode{withOneAMid}{\withOne}
\Rnode{leafTwoBMid}{\leafTwo}
\midgap
\Rnode{leafThreeABMid}{\leafThree}
&
\Rnode{leafOneCMid}{\leafOne}
\Rnode{withOneCMid}{\withOne}
\Rnode{leafTwoDMid}{\leafTwo}
\midgap
\Rnode{leafFourCDMid}{\leafFour}
\ncline{leafOneATop}{leafOneAMid}
\ncline{leafThreeATop}{leafThreeABMid}
\ncline{leafTwoBTop}{leafTwoBMid}
\ncline{leafThreeBTop}{leafThreeABMid}
\ncline{leafOneCTop}{leafOneCMid}
\ncline{leafFourCTop}{leafFourCDMid}
\ncline{leafTwoDTop}{leafTwoDMid}
\ncline{leafFourDTop}{leafFourCDMid}
\\[6ex]
\botshift
\Rnode{leafOneBot}{\leafOne}
\Rnode{withOneBot}{\withOne}
\Rnode{leafTwoBot}{\leafTwo}
&
\Rnode{leafThreeBot}{\leafThree}
\Rnode{withTwoBot}{\withTwo}
\Rnode{leafFourBot}{\leafFour}
\botshift
\ncdiag[angleA=-50,angleB=100,arm=0]{withOneAMid}{withOneBot}
\ncdiag[angleA=-130,angleB=80,arm=0]{withOneCMid}{withOneBot}
\ncline{leafThreeABMid}{leafThreeBot}
\ncline{leafFourCDMid}{leafFourBot}
\end{array}
\hspace*{10ex}
\begin{array}{c@{\hspace*{5ex}}c}
\Rnode{leafOneATop}{\leafOne}
\litgap
\Rnode{leafThreeATop}{\leafThree}
\linkgap
\Rnode{leafOneBTop}{\leafOne}
\litgap
\Rnode{leafFourBTop}{\leafFour}
&
\Rnode{leafTwoCTop}{\leafTwo}
\litgap
\Rnode{leafThreeCTop}{\leafThree}
\linkgap
\Rnode{leafTwoDTop}{\leafTwo}
\litgap
\Rnode{leafFourDTop}{\leafFour}
\ncbar[angle=90,arm=5pt]{leafOneATop}{leafThreeATop}
\ncbar[angle=90,arm=5pt]{leafOneBTop}{leafFourBTop}
\ncbar[angle=90,arm=5pt]{leafTwoCTop}{leafThreeCTop}
\ncbar[angle=90,arm=5pt]{leafTwoDTop}{leafFourDTop}
\\[2ex]
\Rnode{leafOneABMid}{\leafOne}
\midgap
\Rnode{leafThreeBMid}{\leafThree}
\Rnode{withTwoBMid}{\withTwo}
\Rnode{leafFourBMid}{\leafFour}
&
\Rnode{leafTwoCDMid}{\leafTwo}
\midgap
\Rnode{leafThreeDMid}{\leafThree}
\Rnode{withTwoDMid}{\withTwo}
\Rnode{leafFourDMid}{\leafFour}
\ncline{leafOneATop}{leafOneABMid}
\ncline{leafThreeATop}{leafThreeBMid}
\ncline{leafOneBTop}{leafOneABMid}
\ncline{leafFourBTop}{leafFourBMid}
\ncline{leafTwoCTop}{leafTwoCDMid}
\ncline{leafThreeCTop}{leafThreeDMid}
\ncline{leafTwoDTop}{leafTwoCDMid}
\ncline{leafFourDTop}{leafFourDMid}
\\[6ex]
\botshift
\Rnode{leafOneBot}{\leafOne}
\Rnode{withOneBot}{\withOne}
\Rnode{leafTwoBot}{\leafTwo}
&
\Rnode{leafThreeBot}{\leafThree}
\Rnode{withTwoBot}{\withTwo}
\Rnode{leafFourBot}{\leafFour}
\botshift
\ncline{leafOneABMid}{leafOneBot}
\ncline{leafTwoCDMid}{leafTwoBot}
\ncdiag[angleA=-50,angleB=100,arm=0]{withTwoBMid}{withTwoBot}
\ncdiag[angleA=-130,angleB=80,arm=0]{withTwoDMid}{withTwoBot}
\end{array}
\end{math}
\end{center}
Finally, draw nodes instead of formulas, to remove some redundancy, and
where two formulas merge, make that explicit by drawing a contraction
node ($\csym$-node):
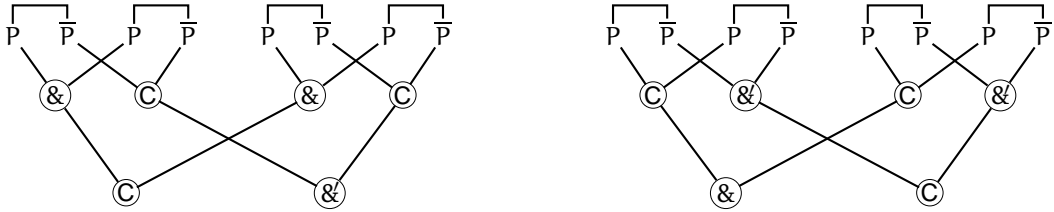
\begin{center}\vspace{2ex}
\begin{math}
\psset{nodesep=1pt}
\newcommand{\litgap}{\hspace*{3ex}}
\newcommand{\linkgap}{\hspace*{4ex}}
\newcommand{\midgap}{\hspace*{5ex}}
\newcommand{\botshift}{\hspace*{4ex}}
\newcommand{\withOneNode}{\pscirclebox[framesep=0pt,linewidth=.1pt]{\withOne}}
\newcommand{\withTwoNode}{\pscirclebox[framesep=0pt,linewidth=.1pt]{\withOne\begin{picture}(0,0)\put(-4,0){${}'$}\end{picture}}}
\begin{array}{c@{\hspace*{5ex}}c}
\Rnode{leafOneATop}{\leafOne}
\litgap
\Rnode{leafThreeATop}{\leafThree}
\linkgap
\Rnode{leafTwoBTop}{\leafTwo}
\litgap
\Rnode{leafThreeBTop}{\leafThree}
&
\Rnode{leafOneCTop}{\leafOne}
\litgap
\Rnode{leafFourCTop}{\leafFour}
\linkgap
\Rnode{leafTwoDTop}{\leafTwo}
\litgap
\Rnode{leafFourDTop}{\leafFour}
\ncbar[angle=90,arm=5pt]{leafOneATop}{leafThreeATop}
\ncbar[angle=90,arm=5pt]{leafTwoBTop}{leafThreeBTop}
\ncbar[angle=90,arm=5pt]{leafOneCTop}{leafFourCTop}
\ncbar[angle=90,arm=5pt]{leafTwoDTop}{leafFourDTop}
\\[2ex]
\Rnode{leafOneAMid}{\Rnode{withOneAMid}{\Rnode{leafTwoBMid}{\withOneNode}}}
\midgap
\Rnode{leafThreeABMid}{\contractionnodeanon}
&
\Rnode{leafOneCMid}{\Rnode{withOneCMid}{\Rnode{leafTwoDMid}{\withOneNode}}}
\midgap
\Rnode{leafFourCDMid}{\contractionnodeanon}
\psset{nodesepB=-1pt}
\ncline{leafOneATop}{leafOneAMid}
\ncline{leafThreeATop}{leafThreeABMid}
\ncline{leafTwoBTop}{leafTwoBMid}
\ncline{leafThreeBTop}{leafThreeABMid}
\ncline{leafOneCTop}{leafOneCMid}
\ncline{leafFourCTop}{leafFourCDMid}
\ncline{leafTwoDTop}{leafTwoDMid}
\ncline{leafFourDTop}{leafFourCDMid}
\\[5ex]
\botshift
\Rnode{leafOneBot}{\Rnode{withOneBot}{\Rnode{leafTwoBot}{\contractionnodeanon}}}
&
\Rnode{leafThreeBot}{\Rnode{withTwoBot}{\Rnode{leafFourBot}{\withTwoNode}}}
\botshift
\psset{nodesep=-1pt}
\ncline{withOneAMid}{withOneBot}
\ncline{withOneCMid}{withOneBot}
\ncline{leafThreeABMid}{leafThreeBot}
\ncline{leafFourCDMid}{leafFourBot}
\end{array}
\hspace*{10ex}
\begin{array}{c@{\hspace*{5ex}}c}
\Rnode{leafOneATop}{\leafOne}
\litgap
\Rnode{leafThreeATop}{\leafThree}
\linkgap
\Rnode{leafTwoBTop}{\leafTwo}
\litgap
\Rnode{leafThreeBTop}{\leafThree}
&
\Rnode{leafOneCTop}{\leafOne}
\litgap
\Rnode{leafFourCTop}{\leafFour}
\linkgap
\Rnode{leafTwoDTop}{\leafTwo}
\litgap
\Rnode{leafFourDTop}{\leafFour}
\ncbar[angle=90,arm=5pt]{leafOneATop}{leafThreeATop}
\ncbar[angle=90,arm=5pt]{leafTwoBTop}{leafThreeBTop}
\ncbar[angle=90,arm=5pt]{leafOneCTop}{leafFourCTop}
\ncbar[angle=90,arm=5pt]{leafTwoDTop}{leafFourDTop}
\\[2ex]
\Rnode{leafOneAMid}{\Rnode{withOneAMid}{\Rnode{leafTwoBMid}{\contractionnodeanon}}}
\midgap
\Rnode{leafThreeABMid}{\withTwoNode}
&
\Rnode{leafOneCMid}{\Rnode{withOneCMid}{\Rnode{leafTwoDMid}{\contractionnodeanon}}}
\midgap
\Rnode{leafFourCDMid}{\withTwoNode}
\psset{nodesepB=-1pt}
\ncline{leafOneATop}{leafOneAMid}
\ncline{leafThreeATop}{leafThreeABMid}
\ncline{leafTwoBTop}{leafTwoBMid}
\ncline{leafThreeBTop}{leafThreeABMid}
\ncline{leafOneCTop}{leafOneCMid}
\ncline{leafFourCTop}{leafFourCDMid}
\ncline{leafTwoDTop}{leafTwoDMid}
\ncline{leafFourDTop}{leafFourCDMid}
\\[5ex]
\botshift
\Rnode{leafOneBot}{\Rnode{withOneBot}{\Rnode{leafTwoBot}{\withOneNode}}}
&
\Rnode{leafThreeBot}{\Rnode{withTwoBot}{\Rnode{leafFourBot}{\contractionnodeanon}}}
\botshift
\psset{nodesep=-1pt}
\ncline{withOneAMid}{withOneBot}
\ncline{withOneCMid}{withOneBot}
\ncline{leafThreeABMid}{leafThreeBot}
\ncline{leafFourCDMid}{leafFourBot}
\end{array}
\end{math}
\end{center}

\subsection{Circuits}

A \defn{circuit} comprises:
\begin{itemize}
\item A finite, non-empty set of \defn{nodes}. 
\item A finite set of \defn{wires}.  Each wire is labelled with a formula, and is 
assigned a \defn{source} node and, possibly, a \defn{target} node.  If
a target node is present, it is distinct from the source node.  A wire with no
target is an \defn{exit}.
\item Each node has one of the following forms:
\begin{itemize}
\item \defn{Axiom}.  The source of
two wires and the target of none.
The wires are labelled by dual literals.\footnote{If we wish to
include cuts, we define a cut node as the target of two wires,
labelled by dual formulas, and the source of no wire.}
\item \defn{Contraction}.  The target of two wires and the source of one.  All three wires have the same formula.
\item \defn{Binary}.  The target of two wires and the source of one.  The incoming wires are distinguished 
as a \defn{left} wire and a \defn{right} wire.  A binary node is typed
as one of $\tensor$, $\parr$ or $\with$.  If the formula of the left
wire is $A$, the formula of the right wire is $B$, and the node type is $\square$, the formula of
the outgoing wire is $A\square B$.
\item \defn{Plus}.  The target of one wire and the source of one wire.  The incoming wire is distinguished 
as \defn{left} or \defn{right}.  Let $A$ be the formula of the incoming wire.  If the incoming wire is left (resp.\ right)
then the formula of the outgoing wire is $A\plus B$ (resp.\ $B\plus A$) for some formula $B$.
\end{itemize}
\item The graph is connected: for any two nodes $N$ and $N'$ there exists a sequence of nodes $N_1\ldots N_k$ 
with $N_1=N$ and $N_k=N'$ ($k\ge 1$) such that for all
$i\in\{1,\ldots,k-1\}$ the nodes $N_i$ and $N_{i+1}$ are joined by a
wire, \ie, there exists a wire whose source is $N_i$ and target is
$N_{i+1}$, or vice versa.\footnote{By dropping connectedness, and
slightly modifying the definition of erasure below, one could choose to
validate the mix rule.}
\item The exits are equipped with a linear order.  The sequent comprising 
the formulas of the exits, in order, is the \defn{conclusion} of the
circuit.
\end{itemize}
An example of a circuit with concluding sequent $P\with P,\dual
P\with\dual P$ is drawn in Figure~\ref{fig-boxless-net}, formalising the last graph in our motivating discusion above.
\begin{figure}
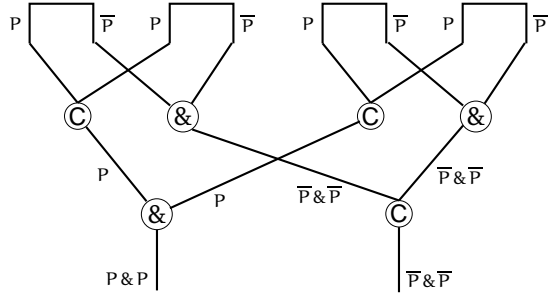

\begin{center}\vspace{2ex}
\begin{math}
\psset{nodesep=1pt}
\newcommand{\litgap}{\hspace*{5ex}}
\newcommand{\linkgap}{\hspace*{6ex}}
\newcommand{\midgap}{\hspace*{6ex}}
\newcommand{\botshift}{\hspace*{4ex}}
\newcommand{\withOneNode}{\pscirclebox[framesep=0pt,linewidth=.1pt]{\withOne}}
\newcommand{\withTwoNode}{\pscirclebox[framesep=0pt,linewidth=.1pt]{\withOne}}
\begin{array}{c@{\hspace*{7ex}}c}
\pnode{leafOneATop}
\litgap
\pnode{leafThreeATop}
\linkgap
\pnode{leafTwoBTop}
\litgap
\pnode{leafThreeBTop}
&
\pnode{leafOneCTop}
\litgap
\pnode{leafFourCTop}
\linkgap
\pnode{leafTwoDTop}
\litgap
\pnode{leafFourDTop}
\psset{nodesep=-1pt}
\ncbar[angle=90,arm=15pt]{leafOneATop}{leafThreeATop}
\wirea{.5}{\leafOne}
\wirea{2.5}{\leafThree}
\ncbar[angle=90,arm=15pt]{leafTwoBTop}{leafThreeBTop}
\wirea{.5}{\leafOne}
\wirea{2.5}{\leafFour}
\ncbar[angle=90,arm=15pt]{leafOneCTop}{leafFourCTop}
\wirea{.5}{\leafTwo}
\wirea{2.5}{\leafThree}
\ncbar[angle=90,arm=15pt]{leafTwoDTop}{leafFourDTop}
\wirea{.5}{\leafTwo}
\wirea{2.5}{\leafFour}
\\[4ex]
\Rnode{leafOneAMid}{\Rnode{withOneAMid}{\Rnode{leafTwoBMid}{\contractionnodeanon}}}
\midgap
\Rnode{leafThreeABMid}{\withTwoNode}
&
\Rnode{leafOneCMid}{\Rnode{withOneCMid}{\Rnode{leafTwoDMid}{\contractionnodeanon}}}
\midgap
\Rnode{leafFourCDMid}{\withTwoNode}
\psset{nodesepB=-1pt}
\ncdiag[angleA=-90,angleB=90,arm=0,nodesepB=0pt]{leafOneATop}{leafOneAMid}
\ncline{leafThreeATop}{leafThreeABMid}
\ncdiag[angleA=-90,angleB=90,arm=0,nodesepB=0pt]{leafTwoBTop}{leafTwoBMid}
\ncline{leafThreeBTop}{leafThreeABMid}
\ncdiag[angleA=-90,angleB=90,arm=0,nodesepB=0pt]{leafOneCTop}{leafOneCMid}
\ncline{leafFourCTop}{leafFourCDMid}
\ncdiag[angleA=-90,angleB=90,arm=0,nodesepB=0pt]{leafTwoDTop}{leafTwoDMid}
\ncline{leafFourDTop}{leafFourCDMid}
\\[5ex]
\botshift
\Rnode{leafOneBot}{\Rnode{withOneBot}{\Rnode{leafTwoBot}{\withOneNode}}}
\nput[labelsep=-7ex]{90}{leafOneBot}{\pnode{conc1}}%
&
\Rnode{leafThreeBot}{\Rnode{withTwoBot}{\Rnode{leafFourBot}{\contractionnodeanon}}}
\nput[labelsep=-7ex]{90}{leafThreeBot}{\pnode{conc2}}%
\botshift
\psset{nodesep=-1pt}
\ncline{withOneAMid}{withOneBot}
\wireb{.5}{\leafOne}
\ncline{withOneCMid}{withOneBot}
\wirea{.75}{\!\!\leafTwo}
\ncdiag[angleA=-55,angleB=90,arm=0,nodesepB=0pt,nodesepA=-1pt]{leafThreeABMid}{leafThreeBot}
\wireb{1.65}{\leafThree\with\leafFour\!\!}
\ncdiag[angleA=-145,angleB=90,arm=0,nodesepB=0pt,nodesepA=-1pt]{leafFourCDMid}{leafFourBot}
\wirea{1.5}{\leafThree\with\leafFour}
\psset{nodesepA=0pt}
\ncline{leafOneBot}{conc1}
\wireb{.8}{\leafOne\with\leafTwo}
\ncline{leafThreeBot}{conc2}
\wirea{.8}{\leafThree\with\leafFour}
\end{array}
\end{math}\vspace*{7ex}
\end{center}
\caption{\label{fig-boxless-net}An example of an erasable circuit (boxless net).}\vspace*{1ex}\hrule
\end{figure}
An axiom node is drawn as a horizontal line segment.  Wires are
oriented downwards in the page (\ie, the target of a wire, when present,
is below its source).  Left/right incoming wires are distinguished by
their contact point being left/right of the centre of the target node.
Contraction nodes are marked $\csym$.  Each wire is labelled with its
formula.  The exits are ordered from left to right in the page.
(The style is similar to interaction nets \cite{Laf90}.)

\subsection{Erasure}

A node is \defn{final} if it is the source of an exit wire.
A node $N$ is \defn{ready} if it is final and it matches one of the
following cases:
\begin{itemize}
\item $N$ is a $\plus$.
\item $N$ is a $\tensor$.  Deleting $N$, and its exit wire, disconnects the circuit
(\ie, the result of deleting $N$ is a disjoint union of two connected components).
\item $N$ is a $\parr$.  Deleting $N$, and its exit wire, does not disconnect the circuit.
\item $N$ is a $\with$.  Every other final node is a contraction-node.
Deleting all final nodes, and their exit wires, yields exactly two
connected components $X_1$ and $X_2$.  Every final node in the
original circuit has one incoming wire
in $X_1$ and the other 
in $X_2$.
\item $N$ is an axiom-node, the unique node of the circuit.
\end{itemize}
Write $X\leadsto_N S$ if $S$ is the set of connected components
resulting from deleting the ready node $N$, each promoted to a circuit
by adding the exit-order induced canonically from the exit-order of
$X$.  By definition of readiness:
\begin{itemize}
\item if $N$ is a $\plus$ or $\parr$ then $S=\{X'\}$, a single
circuit,
\item if $N$ is a $\tensor$, $\with$ or cut-node, then $S=\{X_1,X_2\}$, two circuits.
\item if $N$ is an axiom-node, then $S=\emptyset$, the empty set.
\end{itemize}
If $T$ and $U$ are sets of circuits, write $T\leadsto_{X,N}U$ if
$T=T'\cup\{X\}$ (disjoint union), $X\leadsto_N S$, and $U=T'\cup S$.
(In other words, we replace $X$ by the circuit(s) resulting from
deleting $N$ from $X$.)
Write $T\leadsto U$ if $T\leadsto_{X,N}U$ for some $X$ and $N$.  Note
that $X$ and $N$ are uniquely determined given $T$ and $U$; we call
$N$ the \defn{redex}.
The relation/rewrite $\leadsto$ on sets of circuits is called 
\defn{erasure}.
\begin{proposition}
Erasure $\leadsto$ satisfies the diamond property: if $T\leadsto
U_0$ and $T\leadsto U_1$ with $U_0\neq U_1$, there exists $V$ such
that $U_0\leadsto V$ and $U_1\leadsto V$.
\end{proposition}
\begin{proof}
Suppose $T\leadsto_{X_i,N_i}U_i$.  Assume $X_0=X_1$, or else the result is immediate.  Let $X=X_0=X_1$.
Necessarily $N_0\neq N_1$ (otherwise $U_0=U_1$), therefore $N_i$
cannot be a $\with$-node (since if a $\with$-node is a redex, there
can be no other redex in the same circuit), and cannot be an
axiom-node.  Without loss of generality, ignore cut-node redexes,
since they are homologous to $\tensor$-redexes.  Thus we are left to
consider the following node-types for the redexes $N_0$ and $N_1$: $\parr$,
$\plus$, $\tensor$.
The diamond property is then immediate, since each reduction in these cases
merely deletes a single vertex from a graph.
\end{proof}
Due to more abstract superposition, erasure on conflict nets
will not satisfy the diamond property.  (It will nonetheless be confluent.)
\begin{proposition}
Erasure $\leadsto$ is terminating (strongly normalising).
\end{proposition}
\begin{proof}
If $U\leadsto V$ then the disjoint union of the circuits in $V$ has
strictly less nodes than the disjoint union of the circuits in $U$.
\end{proof}
Write $\leadsto^*$ for the transitive closure of erasure $\leadsto$.
\begin{proposition}
Erasure $\leadsto$ is confluent: if $T\leadsto^*U_0$ and
$T\leadsto^*U_1$ then there exists $V$ such that $U_0\leadsto^*V$ and
$U_1\leadsto^*V$.
\end{proposition}
\begin{proof}
Cut elimination is locally confluent (since it has the diamond
property) and is terminating, so confluence follows from Newman's
lemma \cite{New42}.
\end{proof}
Thus every set of circuits has a unique $\leadsto$-normal form.
A set of circuits $S$ is \defn{erasable} if its normal form is empty,
\ie, if $S\leadsto^*\emptyset$.  
A circuit $X$ is erasable if $\{X\}$ is erasable.
\begin{definition}
A \defn{boxless net} is an erasable circuit.
\end{definition}
Figure~\ref{fig-boxless-net} depicts an example of a boxless net $X$.
An erasure sequence for $X$ is illustrated in Figure~\ref{fig-circuit-erasure}.
\begin{figure}
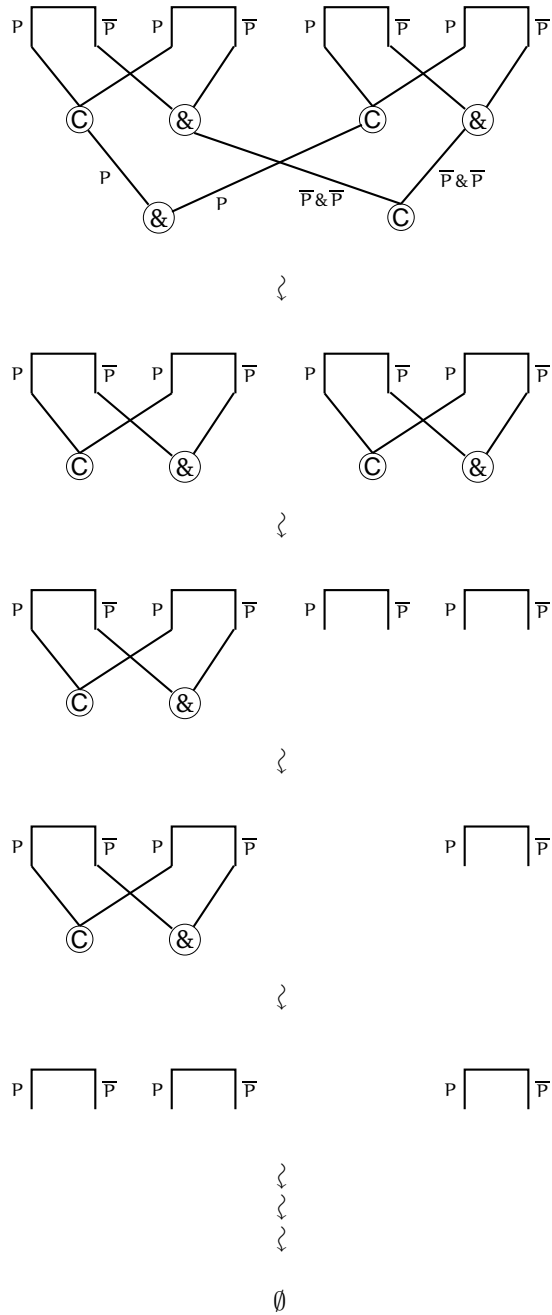

\begin{center}\vspace{2ex}
\begin{math}
\psset{nodesep=1pt}
\newcommand{\litgap}{\hspace*{5ex}}
\newcommand{\linkgap}{\hspace*{6ex}}
\newcommand{\midgap}{\hspace*{6ex}}
\newcommand{\botshift}{\hspace*{4ex}}
\newcommand{\withOneNode}{\pscirclebox[framesep=0pt,linewidth=.1pt]{\withOne}}
\newcommand{\withTwoNode}{\pscirclebox[framesep=0pt,linewidth=.1pt]{\withOne}}
\begin{array}{c@{\hspace*{7ex}}c}
\pnode{leafOneATop}
\litgap
\pnode{leafThreeATop}
\linkgap
\pnode{leafTwoBTop}
\litgap
\pnode{leafThreeBTop}
&
\pnode{leafOneCTop}
\litgap
\pnode{leafFourCTop}
\linkgap
\pnode{leafTwoDTop}
\litgap
\pnode{leafFourDTop}
\psset{nodesep=-1pt}
\ncbar[angle=90,arm=15pt]{leafOneATop}{leafThreeATop}
\wirea{.5}{\leafOne}
\wirea{2.5}{\leafThree}
\ncbar[angle=90,arm=15pt]{leafTwoBTop}{leafThreeBTop}
\wirea{.5}{\leafOne}
\wirea{2.5}{\leafFour}
\ncbar[angle=90,arm=15pt]{leafOneCTop}{leafFourCTop}
\wirea{.5}{\leafTwo}
\wirea{2.5}{\leafThree}
\ncbar[angle=90,arm=15pt]{leafTwoDTop}{leafFourDTop}
\wirea{.5}{\leafTwo}
\wirea{2.5}{\leafFour}
\\[4ex]
\Rnode{leafOneAMid}{\Rnode{withOneAMid}{\Rnode{leafTwoBMid}{\contractionnodeanon}}}
\midgap
\Rnode{leafThreeABMid}{\withTwoNode}
&
\Rnode{leafOneCMid}{\Rnode{withOneCMid}{\Rnode{leafTwoDMid}{\contractionnodeanon}}}
\midgap
\Rnode{leafFourCDMid}{\withTwoNode}
\psset{nodesepB=-1pt}
\ncdiag[angleA=-90,angleB=90,arm=0,nodesepB=0pt]{leafOneATop}{leafOneAMid}
\ncline{leafThreeATop}{leafThreeABMid}
\ncdiag[angleA=-90,angleB=90,arm=0,nodesepB=0pt]{leafTwoBTop}{leafTwoBMid}
\ncline{leafThreeBTop}{leafThreeABMid}
\ncdiag[angleA=-90,angleB=90,arm=0,nodesepB=0pt]{leafOneCTop}{leafOneCMid}
\ncline{leafFourCTop}{leafFourCDMid}
\ncdiag[angleA=-90,angleB=90,arm=0,nodesepB=0pt]{leafTwoDTop}{leafTwoDMid}
\ncline{leafFourDTop}{leafFourCDMid}
\\[5ex]
\botshift
\Rnode{leafOneBot}{\Rnode{withOneBot}{\Rnode{leafTwoBot}{\withOneNode}}}
\nput[labelsep=-7ex]{90}{leafOneBot}{\pnode{conc1}}%
&
\Rnode{leafThreeBot}{\Rnode{withTwoBot}{\Rnode{leafFourBot}{\contractionnodeanon}}}
\nput[labelsep=-7ex]{90}{leafThreeBot}{\pnode{conc2}}%
\botshift
\psset{nodesep=-1pt}
\ncline{withOneAMid}{withOneBot}
\wireb{.5}{\leafOne}
\ncline{withOneCMid}{withOneBot}
\wirea{.75}{\!\!\leafTwo}
\ncdiag[angleA=-55,angleB=90,arm=0,nodesepB=0pt,nodesepA=-1pt]{leafThreeABMid}{leafThreeBot}
\wireb{1.65}{\leafThree\with\leafFour\!\!}
\ncdiag[angleA=-145,angleB=90,arm=0,nodesepB=0pt,nodesepA=-1pt]{leafFourCDMid}{leafFourBot}
\wirea{1.5}{\leafThree\with\leafFour}
\end{array}
\end{math}

\vspace*{4ex}
\leadsdownto
\vspace*{6ex}

\begin{math}
\psset{nodesep=1pt}
\newcommand{\litgap}{\hspace*{5ex}}
\newcommand{\linkgap}{\hspace*{6ex}}
\newcommand{\midgap}{\hspace*{6ex}}
\newcommand{\botshift}{\hspace*{4ex}}
\newcommand{\withOneNode}{\pscirclebox[framesep=0pt,linewidth=.1pt]{\withOne}}
\newcommand{\withTwoNode}{\pscirclebox[framesep=0pt,linewidth=.1pt]{\withOne}}
\begin{array}{c@{\hspace*{7ex}}c}
\pnode{leafOneATop}
\litgap
\pnode{leafThreeATop}
\linkgap
\pnode{leafTwoBTop}
\litgap
\pnode{leafThreeBTop}
&
\pnode{leafOneCTop}
\litgap
\pnode{leafFourCTop}
\linkgap
\pnode{leafTwoDTop}
\litgap
\pnode{leafFourDTop}
\psset{nodesep=-1pt}
\ncbar[angle=90,arm=15pt]{leafOneATop}{leafThreeATop}
\wirea{.5}{\leafOne}
\wirea{2.5}{\leafThree}
\ncbar[angle=90,arm=15pt]{leafTwoBTop}{leafThreeBTop}
\wirea{.5}{\leafOne}
\wirea{2.5}{\leafFour}
\ncbar[angle=90,arm=15pt]{leafOneCTop}{leafFourCTop}
\wirea{.5}{\leafTwo}
\wirea{2.5}{\leafThree}
\ncbar[angle=90,arm=15pt]{leafTwoDTop}{leafFourDTop}
\wirea{.5}{\leafTwo}
\wirea{2.5}{\leafFour}
\\[4ex]
\Rnode{leafOneAMid}{\Rnode{withOneAMid}{\Rnode{leafTwoBMid}{\contractionnodeanon}}}
\midgap
\Rnode{leafThreeABMid}{\withTwoNode}
&
\Rnode{leafOneCMid}{\Rnode{withOneCMid}{\Rnode{leafTwoDMid}{\contractionnodeanon}}}
\midgap
\Rnode{leafFourCDMid}{\withTwoNode}
\psset{nodesepB=-1pt}
\ncdiag[angleA=-90,angleB=90,arm=0,nodesepB=0pt]{leafOneATop}{leafOneAMid}
\ncline{leafThreeATop}{leafThreeABMid}
\ncdiag[angleA=-90,angleB=90,arm=0,nodesepB=0pt]{leafTwoBTop}{leafTwoBMid}
\ncline{leafThreeBTop}{leafThreeABMid}
\ncdiag[angleA=-90,angleB=90,arm=0,nodesepB=0pt]{leafOneCTop}{leafOneCMid}
\ncline{leafFourCTop}{leafFourCDMid}
\ncdiag[angleA=-90,angleB=90,arm=0,nodesepB=0pt]{leafTwoDTop}{leafTwoDMid}
\ncline{leafFourDTop}{leafFourCDMid}
\end{array}
\end{math}

\vspace*{3ex}
\leadsdownto
\vspace*{6ex}

\begin{math}
\psset{nodesep=1pt}
\newcommand{\litgap}{\hspace*{5ex}}
\newcommand{\linkgap}{\hspace*{6ex}}
\newcommand{\midgap}{\hspace*{6ex}}
\newcommand{\botshift}{\hspace*{4ex}}
\newcommand{\withOneNode}{\pscirclebox[framesep=0pt,linewidth=.1pt]{\withOne}}
\newcommand{\withTwoNode}{\pscirclebox[framesep=0pt,linewidth=.1pt]{\withOne}}
\begin{array}{c@{\hspace*{7ex}}c}
\pnode{leafOneATop}
\litgap
\pnode{leafThreeATop}
\linkgap
\pnode{leafTwoBTop}
\litgap
\pnode{leafThreeBTop}
&
\pnode{leafOneCTop}
\litgap
\pnode{leafFourCTop}
\linkgap
\pnode{leafTwoDTop}
\litgap
\pnode{leafFourDTop}
\psset{nodesep=-1pt}
\ncbar[angle=90,arm=15pt]{leafOneATop}{leafThreeATop}
\wirea{.5}{\leafOne}
\wirea{2.5}{\leafThree}
\ncbar[angle=90,arm=15pt]{leafTwoBTop}{leafThreeBTop}
\wirea{.5}{\leafOne}
\wirea{2.5}{\leafFour}
\ncbar[angle=90,arm=15pt]{leafOneCTop}{leafFourCTop}
\wirea{.5}{\leafTwo}
\wirea{2.5}{\leafThree}
\ncbar[angle=90,arm=15pt]{leafTwoDTop}{leafFourDTop}
\wirea{.5}{\leafTwo}
\wirea{2.5}{\leafFour}
\\[4ex]
\Rnode{leafOneAMid}{\Rnode{withOneAMid}{\Rnode{leafTwoBMid}{\contractionnodeanon}}}
\midgap
\Rnode{leafThreeABMid}{\withTwoNode}
\psset{nodesepB=-1pt}
\ncdiag[angleA=-90,angleB=90,arm=0,nodesepB=0pt]{leafOneATop}{leafOneAMid}
\ncline{leafThreeATop}{leafThreeABMid}
\ncdiag[angleA=-90,angleB=90,arm=0,nodesepB=0pt]{leafTwoBTop}{leafTwoBMid}
\ncline{leafThreeBTop}{leafThreeABMid}
\end{array}
\end{math}

\vspace*{3ex}
\leadsdownto
\vspace*{6ex}

\begin{math}
\psset{nodesep=1pt}
\newcommand{\litgap}{\hspace*{5ex}}
\newcommand{\linkgap}{\hspace*{6ex}}
\newcommand{\midgap}{\hspace*{6ex}}
\newcommand{\botshift}{\hspace*{4ex}}
\newcommand{\withOneNode}{\pscirclebox[framesep=0pt,linewidth=.1pt]{\withOne}}
\newcommand{\withTwoNode}{\pscirclebox[framesep=0pt,linewidth=.1pt]{\withOne}}
\begin{array}{c@{\hspace*{7ex}}c}
\pnode{leafOneATop}
\litgap
\pnode{leafThreeATop}
\linkgap
\pnode{leafTwoBTop}
\litgap
\pnode{leafThreeBTop}
&
\pnode{leafOneCTop}
\litgap
\pnode{leafFourCTop}
\linkgap
\pnode{leafTwoDTop}
\litgap
\pnode{leafFourDTop}
\psset{nodesep=-1pt}
\ncbar[angle=90,arm=15pt]{leafOneATop}{leafThreeATop}
\wirea{.5}{\leafOne}
\wirea{2.5}{\leafThree}
\ncbar[angle=90,arm=15pt]{leafTwoBTop}{leafThreeBTop}
\wirea{.5}{\leafOne}
\wirea{2.5}{\leafFour}
\ncbar[angle=90,arm=15pt]{leafTwoDTop}{leafFourDTop}
\wirea{.5}{\leafTwo}
\wirea{2.5}{\leafFour}
\\[4ex]
\Rnode{leafOneAMid}{\Rnode{withOneAMid}{\Rnode{leafTwoBMid}{\contractionnodeanon}}}
\midgap
\Rnode{leafThreeABMid}{\withTwoNode}
\psset{nodesepB=-1pt}
\ncdiag[angleA=-90,angleB=90,arm=0,nodesepB=0pt]{leafOneATop}{leafOneAMid}
\ncline{leafThreeATop}{leafThreeABMid}
\ncdiag[angleA=-90,angleB=90,arm=0,nodesepB=0pt]{leafTwoBTop}{leafTwoBMid}
\ncline{leafThreeBTop}{leafThreeABMid}
\end{array}
\end{math}

\vspace*{3ex}
\leadsdownto
\vspace*{6ex}

\begin{math}
\psset{nodesep=1pt}
\newcommand{\litgap}{\hspace*{5ex}}
\newcommand{\linkgap}{\hspace*{6ex}}
\newcommand{\midgap}{\hspace*{6ex}}
\newcommand{\botshift}{\hspace*{4ex}}
\newcommand{\withOneNode}{\pscirclebox[framesep=0pt,linewidth=.1pt]{\withOne}}
\newcommand{\withTwoNode}{\pscirclebox[framesep=0pt,linewidth=.1pt]{\withOne}}
\begin{array}{c@{\hspace*{7ex}}c}
\pnode{leafOneATop}
\litgap
\pnode{leafThreeATop}
\linkgap
\pnode{leafTwoBTop}
\litgap
\pnode{leafThreeBTop}
&
\pnode{leafOneCTop}
\litgap
\pnode{leafFourCTop}
\linkgap
\pnode{leafTwoDTop}
\litgap
\pnode{leafFourDTop}
\psset{nodesep=-1pt}
\ncbar[angle=90,arm=15pt]{leafOneATop}{leafThreeATop}
\wirea{.5}{\leafOne}
\wirea{2.5}{\leafThree}
\ncbar[angle=90,arm=15pt]{leafTwoBTop}{leafThreeBTop}
\wirea{.5}{\leafOne}
\wirea{2.5}{\leafFour}
\ncbar[angle=90,arm=15pt]{leafTwoDTop}{leafFourDTop}
\wirea{.5}{\leafTwo}
\wirea{2.5}{\leafFour}
\end{array}
\end{math}

\vspace*{3ex}
\leadsdownto

\leadsdownto

\leadsdownto
\vspace*{3ex}

$\emptyset$

\end{center}
\caption{\label{fig-circuit-erasure}An erasure sequence.  
To save space, we leave exit wires from final $\with$ and $\csym$
nodes implied.}
\end{figure}
Note that, by the diamond property, any erasure sequence from $X$ to
$\emptyset$ has the same number of steps: the number of
non-contraction nodes in $X$.

\subsection{P-time correctness}

The following theorem distinguishes erasability from mere
sequentializability.
\begin{theorem}
Erasability of a circuit $X$ can be checked in p-time in the number of
nodes in $X$.
\end{theorem}
\begin{proof}
Let $k$ be the number of nodes in $X$, and $n$ the number of
non-contraction nodes.  Since each erasure step deletes a
non-contraction node, the $\leadsto$-normal form of $\{X\}$ is
obtained in at most $n$ steps.
By the diamond property, any ready node $N$ suffices at each step.
To find such an $N$ requires checking at most $n$ nodes for readiness,
and the complexity of checking if a node is ready is at worst the
complexity of checking disconnectedness of a graph $G$ into two
connected components, where $G$ has at most $k$ vertices.
\end{proof}

\subsection{Translation function from proofs to circuits}

The obvious translation via box nets was outlined at the beginning of
the section: simply forget to draw the boxes.
For the sake of complete rigour, we give below a direct formal
translation of a proof $\Pi$ to a circuit $X$, by induction on the
number of rules in $\Pi$.
\begin{itemize}
\item \emph{Base case.} $\Pi$ is an axiom with conlusion $P,\dual P$.  $X$ is an axiom-node $N$ two exit wires, labelled $P$ and $\dual P$, 
in that order.
\item \emph{Induction step.}  Let $\rho$ be the last rule of $\Pi$, and $\Gamma$ its conclusion.
\begin{itemize}
\item 
\emph{Unary case.} 
$\rho$ has one hypothesis sequent $\Delta$ above its line, which
concludes the subproof $\Pi'$ of $\Pi$.  Let $X'$ be the circuit
obtained from $\Pi'$.
\begin{itemize}
\item $\rho=\permlabel_\sigma$.
Define $X$ from $X'$ by applying the permutation $\sigma$ to the
ordering of the exit wires (viewing the ordering as an enumeration
from $1$).
\item
$\rho=\parr$, so $\Delta=\Delta',A,B$ and $\Gamma=\Delta',A\parr B$.
Define $X$ from $X'$ as follows: add a new $\parr$-node $N$ as the
target of the last two exit wires of $X'$ (the last wire being
designated \emph{right} for $N$); add to $N$ a new exit wire $w$
labelled $A\parr B$; place $w$ in last position in the exit wire
order.
\item $\rho=\plus_i$, so $\Delta=\Delta',A_i$ and $\Gamma=\Delta',A_0\plus A_1$.  Define $X$ from $X'$ as follows: add
a new $\plus$-node $N$ as the target of the last wire $v$ of $X'$, and
designate $v$ as left or right according to $i=0$ or $1$; add to $N$ a
new exit wire $w$ labelled $A_0\parr A_1$; place place $w$ in last
position in the exit wire order.
\end{itemize}
\item \emph{Binary case.}
$\rho$ has two hypotheses $\Delta_0$ and $\Delta_1$, which conclude
subproofs $\Pi_0$ and $\Pi_1$ of $\Pi$, respectively.  Let $X_i$ be
the circuit obtained from $\Pi_i$.
\begin{itemize}
\item $\rho=\tensor$, so $\Delta_0=\Delta'_0,A$ and $\Delta_1=B,\Delta_1'$.  Define $X$ from the disjoint union 
of $X_0$ and $X_1$ as follows: add a new $\tensor$-node $N$ as the
target of the last wire $v_0$ of $X_0$ and the first wire $v_1$ of
$X_1$; designate $v_0$ as left for $N$ and $v_1$ as right; add to $N$
a new exit wire $w$ labelled $A\tensor B$; impose the following order
on exit wires: all the exit wires of $X_0$ in their original order
(except $v_0$, which is no longer an exit), then $w$, then all the
exit wires of $X_1$ in their original order (except $v_1$, which is no
longer an exit).
\item $\rho=\with$, so $\Delta_i=\Delta',A_i$.  Let $\Delta'=B_1,\ldots,B_n$.  Define $X$ from the disjoint union 
of $X_0$ and $X_1$ as follows: add a new $\with$-node $N$ as the
target of the last wire $v_0$ of $X_0$ and the last wire $v_1$ of
$X_1$; designate $v_0$ as left for $N$ and $v_1$ as right; add to $N$
a new exit wire $w$ labelled $A\tensor B$; for $j=1,\ldots,n$ add a
new contraction-node $N_j$ as the target of the $j\nth$ wire of $X_0$
and the $j\nth$ wire of $X_1$; add to $N_j$ a new exit wire $w_j$
labelled $B_j$; impose the following order on exit wires:
$w_1,\ldots,w_n,w$.
\end{itemize}
\end{itemize}
\end{itemize}
\begin{proposition}\label{prop-circuit-soundness}
The above translation maps every proof to an erasable circuit.
\end{proposition}
\begin{proof}
By induction on the number of rules in the proof $\Pi$.
We reference each case in the translation above:
\begin{itemize}
\item \emph{Base case.}
$X$ is erasable in one step: $\{X\}\leadsto_{X,N}\emptyset$.
\item \emph{Induction step.}
\begin{itemize}
\item $\rho=\permlabel_\sigma$.
The circuits $X$ and $X'$ differ only in the order on their exit
wires.  Since node readiness is independent of exit wire order,
$X$ is erasable by the same sequence of erasures as $X'$.
\item $\rho=\parr$ or $\plus$.
$\{X\}\leadsto_{X,N} \{X'\}$ by construction, and $X'$ is erasable.
\item $\rho=\tensor$ or $\plus$.
$\{X\}\leadsto_{X,N} \{X_1,X_2\}$ by construction, and each $X_i$ is
erasable.  Thu $X$ is erasable by (arbitrarily) interleaving erasure
sequences of $X_1$ and $X_2$ after $\{X\}\leadsto_{X,N}\{X_1,X_2\}$.
\end{itemize}
\end{itemize}\vspace*{-4ex}
\end{proof}
A circuit $X$ is \defn{sequentializable} if it is the translation of a
proof; any such proof is a \defn{sequentialization} of $X$.

\subsection{Sequentialization}

\begin{theorem}[Sequentialization]\label{thm-boxless-seq}
A circuit is erasable iff it is sequentializable.
\end{theorem}
\begin{proof}
The right-to-left implication is Proposition~\ref{prop-circuit-soundness}.

Let $X$ be an erasable circuit, with $n$-step erasure sequence to
$\emptyset$.  We prove $X$ sequentializable by induction on $n$ (which
is the same for all erasure sequences to $\emptyset$, by the diamond
property).
\begin{itemize}
\item
\emph{Base case.}
$n=1$. $X$ is the translation of an axiom rule.
\item
\emph{Inductive step.}
$n>1$.  Let $N$ be the ready node deleted from $X$ in the first
erasure step.  
Let $v_1,\ldots,v_n$ be the exit wires of $X$, in order, and let $C_i$
be the formula of $v_i$.
Suppose $v_k$ be the exit wire of $N$ ($1\le k\le n$) and let
$\Gamma_1=C_1,\ldots,C_{k-1}$ and $\Gamma_2=C_{k+1}\ldots C_n$.
We split into subcases according to the type of $N$.
\begin{itemize}
\item
\emph{Unary case.} $N$ is a $\parr$ or $\plus$.  Thus $\{X\}\leadsto_{X,N} \{Y\}$ is the
first erasure step.  By induction hypothesis, a proof $\Pi$ translates
to $Y$. 
\begin{itemize}
\item $N$ is a $\parr$.
Let $A$ be the formula of the left incoming
wire of $N$, and $B$ the formula of the right incoming
wire.  
The following proof translates to $X$:
\begin{center}
\vspace{1ex}
\begin{prooftree}\thickness=.08em
\[
  \Pi
  \justifies
  \Gamma_1,A,B,\Gamma_2
\]
\justifies
\Gamma_1,A\parr B,\Gamma_2
\using 
\parlabel
\end{prooftree}
\vspace{1ex}
\end{center}
(Permutation rules are supressed; see Section~\ref{sec-mall}.)
\item $N$ is a $\plus$.
Thus the formula $C_k$ of $N's$ exit wire $v_k$ is $A_0\plus A_1$.
The following proof translates to $X$, where $j=0/1$ according as the incoming wire of $N$ is designated left/right.
\begin{center}
\vspace{1ex}
\begin{prooftree}\thickness=.08em
\[
  \Pi
  \justifies
  \Gamma_1\;,\;A_i\;,\;\Gamma_2
\]
\justifies
\Gamma_1,A_0\plus A_1,\Gamma_2
\using 
\pluslabel{i}
\end{prooftree}
\vspace{1ex}
\end{center}
\end{itemize}
\item
\emph{Binary case.}
$N$ is a $\tensor$ or $\with$.  Thus $\{X\}\leadsto_{X,N}
\{Y_0,Y_1\}$ is the first erasure step.  By induction hypothesis, proofs
$\Pi_i$ translate to $Y_i$.
Let $u_0$ be the left incoming wire of $N$, labelled $A_0$, and $u_1$
its right incoming wire, labelled $A_1$.
\begin{itemize}
\item
$N$ is a $\tensor$. Thus $C_k=A_0\tensor A_1$.
The conclusion of $\Pi_i$ is $\Delta_i,A_i,\Delta'_i$.
The following proof translates to $X$:
\begin{center}
\vspace{1ex}
\begin{prooftree}\thickness=.08em
\[
  \[
    \[
      \Pi_0
      \justifies
      \Delta_0,A_0,\Delta_0'
    \]
    \justifies
    \Delta_0,\Delta_0',A_0
    \using\permlabel
  \]
  \hspace*{6ex}
  \[ 
    \[
      \Pi_1
      \justifies
      \Delta_1,A_1,\Delta_1'
    \]
    \justifies
    A_1,\Delta_1,\Delta_1'
    \using\permlabel
  \]  
  \justifies
  \Delta_0,\Delta'_0,A_0\tensor A_1,\Delta_1,\Delta_1'
  \using\tensorlabel
\]
\justifies
\Gamma_1,A_0\tensor A_1,\Gamma_2
\using 
\permlabel
\end{prooftree}
\vspace{1ex}
\end{center}
The permutations are determined by the fact that the exit wires of
$Y_0$ and $Y_1$ apart from $u_0$ and $u_1$ are exactly the exit wires
of $X$ apart from $w_k$.
\item
$N$ is a $\with$. Thus $C_k=A_0\with A_1$.
The conclusion of $\Pi_i$ is $\Gamma_1,A_i,\Gamma_2$.
The following proof translates to $X$:
\begin{center}
\vspace{1ex}
\begin{prooftree}\thickness=.08em
\[
  \[
    \[
      \Pi_0
      \justifies
      \Gamma_1,A_0,\Gamma_2
    \]
    \justifies
    \Gamma_1,\Gamma_2,A_1
    \using\permlabel
  \]
  \hspace*{6ex}
  \[ 
    \[
      \Pi_1
      \justifies
      \Gamma_1,A_1,\Gamma_2
    \]
    \justifies
    \Gamma_1,\Gamma_2,A_1
    \using\permlabel
  \]  
  \justifies
  \Gamma_1,\Gamma_2,A_0\with A_1
  \using\withlabel
\]
\justifies
\Gamma_1,A_0\with A_1,\Gamma_2
\using 
\permlabel
\end{prooftree}
\vspace{1ex}
\end{center}
The permutations are determined by the bijections between the exit
wires of each $Y_i$ and the exit wires of $X$.
\end{itemize}
\end{itemize}
\end{itemize}\vspace*{-4ex}
\end{proof}

\subsection{Relationship with contractibility/retractability}

The underlying data structure of a circuit (aside from the order on
exit wires, which is a technical convenience) is the same as that used
by Maieli \cite{Mai07}.
\begin{conjecture}
A circuit is the translation of a proof iff it is retractable with
respect to Maieli's $R_1,\ldots,R_4$ (dropping $R_5$).
\end{conjecture}

\section{Erasure for conflict nets}\label{sec-erasure}

We can draw a linking $\lambda : L\to\Gamma$ as a graph in two
different ways, depending on whether we show conflict $\conflict$ or
adjacency $\girstcoh$.  For example, the linking below is
followed by each of its graphs, the former graph showing conflict
$\conflict$ (dotted), the latter showing adjacency $\girstcoh$
(dashed).
The three links are shown as $\bullet$ vertices.
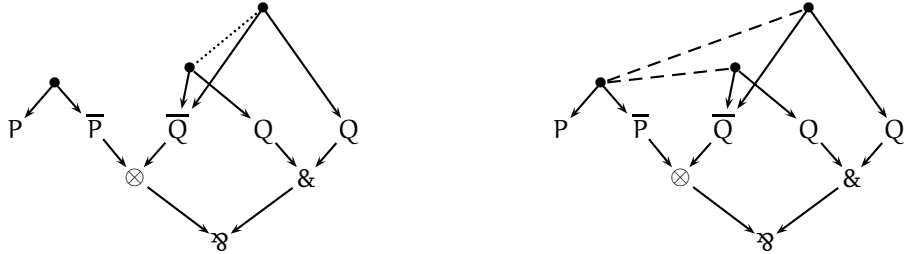
\begin{center}%
\vspace*{3ex}%
\newcommand{\gap}{\hspace{5ex}}%
\newcommand{\sequentOne}{\begin{math}%
\Rnode{P}{P}
\gap
(\Rnode{P'}{\dual P}
\tensor
\Rnode{Q'}{\dual Q})
\parr
(\Rnode{Q1}{Q}
\with
\Rnode{Q2}{Q})
\end{math}}%
\newcommand{\sequentTwo}{\begin{math}%
\Rnode{P}{P}
\gap
\Rnode{P'}{\dual P}
\tensor
\Rnode{Q'}{\dual Q}
\gap
\Rnode{Q1}{Q}
\with
\Rnode{Q2}{Q}
\end{math}}%
\raisebox{-2ex}{\sequentOne%
\ncbar[angle=90,nodesep=2pt,arm=.4cm]{P}{P'}
\ncbar[angle=90,nodesep=2pt,arm=.2cm,offsetA=-1pt]{Q'}{Q1}
\ncput{\pnode{Q'Q1}}
\ncbar[angle=90,nodesep=2pt,arm=.6cm,offsetA=1pt]{Q'}{Q2}
\ncput{\pnode{Q'Q2}}
\conflictline{Q'Q1}{Q'Q2}}

\vspace*{13ex}

\raisebox{2.5ex}{\begin{math}%
\psset{nodesepA=.5pt,nodesepB=.5pt}%
\Rnode{P}{P}
\gap
\Rnode{P'}{\dual P}
\gap
\Rnode{Q'}{\dual Q}
\gap
\Rnode{Q1}{Q}
\gap
\Rnode{Q2}{Q}
\ncline[linestyle=none]{P'}{Q'}
\nbput[labelsep=.5cm]{\rnode{t}{\tensor}}
\ncline{->}{P'}{t}
\ncline{->}{Q'}{t}
\ncline[linestyle=none]{Q1}{Q2}
\nbput[labelsep=.5cm]{\rnode{w}{\with}}
\ncline{->}{Q1}{w}
\ncline{->}{Q2}{w}
\ncline[linestyle=none]{P}{P'}
\naput[labelsep=.5cm]{\rnode{l}{\bullet}}
\ncline[nodesepA=-1pt]{->}{l}{P}
\ncline[nodesepA=-1pt,nodesepB=2pt]{->}{l}{P'}
\ncline[linestyle=none]{Q'}{Q2}
\naput[npos=0,labelsep=.7cm]{\rnode{m}{\bullet}}
\ncline[nodesepA=-1pt,nodesepB=2pt]{->}{m}{Q'}
\ncline[nodesepA=-1pt]{->}{m}{Q1}
\ncline[linestyle=none]{Q'}{Q2}
\naput[npos=.5,labelsep=1.5cm]{\rnode{n}{\bullet}}
\ncline[nodesepA=-1pt,nodesepB=2pt]{->}{n}{Q'}
\ncline[nodesepA=-1pt]{->}{n}{Q2}
\ncline[linestyle=none]{t}{w}
\nbput[labelsep=.7cm]{\rnode{p}{\parr}}
\ncline{->}{t}{p}
\ncline{->}{w}{p}
\conflictline[nodesep=0pt]{m}{n}
\end{math}}
\hspace*{14ex}
\raisebox{2.5ex}{\begin{math}%
\psset{nodesepA=.5pt,nodesepB=.5pt}%
\Rnode{P}{P}
\gap
\Rnode{P'}{\dual P}
\gap
\Rnode{Q'}{\dual Q}
\gap
\Rnode{Q1}{Q}
\gap
\Rnode{Q2}{Q}
\ncline[linestyle=none]{P'}{Q'}
\nbput[labelsep=.5cm]{\rnode{t}{\tensor}}
\ncline{->}{P'}{t}
\ncline{->}{Q'}{t}
\ncline[linestyle=none]{Q1}{Q2}
\nbput[labelsep=.5cm]{\rnode{w}{\with}}
\ncline{->}{Q1}{w}
\ncline{->}{Q2}{w}
\ncline[linestyle=none]{P}{P'}
\naput[labelsep=.5cm]{\rnode{l}{\bullet}}
\ncline[nodesepA=-1pt]{->}{l}{P}
\ncline[nodesepA=-1pt,nodesepB=2pt]{->}{l}{P'}
\ncline[linestyle=none]{Q'}{Q2}
\naput[npos=0,labelsep=.7cm]{\rnode{m}{\bullet}}
\ncline[nodesepA=-1pt,nodesepB=2pt]{->}{m}{Q'}
\ncline[nodesepA=-1pt]{->}{m}{Q1}
\ncline[linestyle=none]{Q'}{Q2}
\naput[npos=.5,labelsep=1.5cm]{\rnode{n}{\bullet}}
\ncline[nodesepA=-1pt,nodesepB=2pt]{->}{n}{Q'}
\ncline[nodesepA=-1pt]{->}{n}{Q2}
\ncline[linestyle=none]{t}{w}
\nbput[labelsep=.7cm]{\rnode{p}{\parr}}
\ncline{->}{t}{p}
\ncline{->}{w}{p}
\ncline[linestyle=dashed,nodesep=0pt]{l}{m}
\ncline[linestyle=dashed,nodesep=0pt]{l}{n}
\end{math}}
\vspace*{6.5ex}
\end{center}
We shall write $\conflictgraph\lambda$ for the left graph, and $\adjacencygraph\lambda$ for the right graph.
Formally,
\begin{eqnarray*}
\conflictgraph\lambda
& \;\;=\;\; & 
\Gamma \;\;\cup\;\; \conflictgraph L \;\;\cup\;\; \lambda
\\
\adjacencygraph\lambda
& \;\;=\;\; & 
\Gamma \;\;\cup\;\; \adjacencygraph L \;\;\cup\;\; \lambda
\end{eqnarray*}
where $L^{\conflict}$ (resp.\ $\adjacencygraph{L}$) denotes the
undirected graph on the links of $L$ given by conflict (resp.\
adjacency), and (without loss of generality) we assume $\Gamma$ and
$L$ are disjoint.
Thus $\lambda^{\conflict}$ is the union of the sequent $\Gamma$
(formula parse trees) and the $\conflict$-graph of $L$, together with
an edge $l\edge x$ whenever $\langle l,x\rangle\in\lambda$ (\ie,
whenever $x$ is a leaf in the dual pair of $l$).

A vertex in a sequent with no outgoing edge is a \defn{root}, and is
said to be \defn{final}.
Let $\diamond\in\{\with,\plus\}$ and let $r$ be the
$\diamond$-labelled root of the formula $A_0\diamond A_1$ in $\Gamma$.
A slicing $\lambda : L\to \Gamma$ \defn{touches} $A_i$ if some leaf of
$A_i$ is in the image of $\lambda$, and \defn{chooses} $A_i$ if it
touches $A_i$ but does not touch $A_{1-i}$.  (Since $\lambda$ is a
slicing, if it is non-empty it must touch at least one of $A_0$ and
$A_1$ by Proposition~\ref{prop-slices}; it is possible that $\lambda$
touches both.)  If $\lambda$ touches exactly one of the $A_i$ we say
that $r$ is \defn{unary} under $\lambda$.
A
\defn{piece}
of $\lambda$ is its restriction to a connected
component\footnote{By convention, a connected component is non-empty.}
of the $\adjacent$-graph $\adjacencygraph{L}$ of $L$.
A slicing $\lambda : L\to\Gamma$ is \defn{connected} if it is
non-empty and its $\conflict$-graph $\conflictgraph\lambda$ is
connected.

Let $\lambda : L\to \Gamma$ be a connected slicing.
A $\square$-labelled root $r$ 
is \defn{ready} in $\lambda$ if one of the following cases holds:
\begin{itemize}
\item
$\square=\tensor$ and $r$ is not in a cycle in
$\conflictgraph{\lambda}$.\footnote{In other words, upon deleting $r$
(and its two incoming edges) there are two connected components.}
\item 
$\square=\parr$.
\item
$\square=\plus$ and $r$ is unary under $\lambda$.
\item
$\square=\with$ and $r$ is unary under every piece of $\lambda$.
\end{itemize}
Let $A_0\square A_1$ be the formula whose root is $r$.  The result of
\defn{erasing} $r$, if $r$ is ready, is a set of slicings $\lambda\setminus r$:
\begin{itemize}
\item
$\square=\tensor$.
Let $\conflictgraph{\lambda}_0$ and $\conflictgraph{\lambda}_1$ be the
connected components of $\conflictgraph\lambda$ upon deleting $r$.
This yields two slicings $\lambda_0$ and $\lambda_1$, the former on a
sequent $\Delta_0,A_0,\Delta_0'$ and the latter on
$\Delta_1,A_1,\Delta_1'$. 
Define $\lambda\setminus r=\{\lambda_0,\lambda_1\}$.
\item 
$\square=\parr$.  Let $\conflictgraph{\lambda_0}$ be the result of
deleting $r$ from $\conflictgraph\lambda$, yielding a slicing $\lambda_0$.
Define $\lambda\setminus r=\{\lambda_0\}$.
\item
$\square=\plus$.  Since $r$ is unary under $\lambda$ and $\lambda$ is
non-empty, $\lambda$ chooses $A_j$ for some $i\in\{0,1\}$.
Let $\conflictgraph\lambda_j$ be the result of deleting $r$ and
$A_{1-j}$ from $\conflictgraph\lambda$, yielding a slicing
$\lambda_j$.
Define $\lambda\setminus r=\{\lambda_j\}$.
\item
$\square=\with$.
Let $\Gamma=\Delta,A_0\with A_1,\Sigma$.  Let $\lambda_i$ be the
slicing on $\Delta,A_i,\Sigma$ comprising the union of all pieces
of $\lambda$ which choose $A_i$.  
Define $\lambda\setminus r=\{\lambda_0,\lambda_1\}$.
(By Proposition~\ref{prop-slices}, every piece of $\lambda$ chooses one of the $A_i$.  Thus $\lambda=\lambda_0\cup\lambda_1$.)
\end{itemize}
Note that even though $\lambda$ is connected, a slicing in
$\lambda\setminus r$ may be disconnected (\eg\ empty).

A \defn{cluster} is either a set of slicings or the \defn{error} symbol $\error$.
Define \defn{erasure} $\leadsto$ on clusters as follows.
\begin{itemize}
\item $Y\leadsto\error$ if $Y$ contains a slicing which is disconnected. (Note: any empty slicing is disconnected.)
\item $X\cup\{\lambda\}\,\leadsto\,X\cup(\lambda\setminus r)$ if $r$ is a ready root of $\lambda$,
and every slicing in $X$ is connected.  
Here we assume $\lambda\not\in X$.
\item $X\cup\{\lambda\}\,\leadsto\,X$ if $\lambda$ is a single link on $P,\dual P$ for some literal $P$ 
(\ie, if $\lambda$ corresponds to an axiom), and every slicing of $X$ is connected.  Here we assume $\lambda\not\in X$.
\end{itemize}
Write $\leadsto^*$ for the transitive closure of $\leadsto$.
\begin{proposition}
Erasure $\leadsto$ is locally confluent (weak Church-Rosser): if $X\leadsto
Y_0$ and $X\leadsto Y_1$ there exists a cluster $Z$ such that
$Y_0\leadsto^* Z$ and $Y_1\leadsto^* Z$.
\end{proposition}
\begin{proof}
Suppose $X\leadsto Y_i$ by erasing $r_i$ from $\lambda_i\in X$.
Assume $\lambda_0=\lambda_1$, or else the result is immediate.
Let $\lambda=\lambda_0=\lambda_1$.
Assume $r_0\neq r_1$, otherwise the result holds with $Z=Y_0=Y_1$.
Let $\square_i$ be the connective of $r_i$.  We split cases according to $\square_0$.
\begin{itemize}
\item $\square_0=\tensor$.  
Let $\lambda\setminus r_0=\{\lambda_a,\lambda_b\}$, with both
$\lambda_a$ and $\lambda_b$ connected. Without loss of generality, assume
$r_1$ is in the sequent of $\lambda_a$.
We split cases according to $\square_1$.
\begin{itemize}
\item $\square_1=\plus$ or $\parr$. Then $\lambda_a\setminus
r_1=\{\lambda_a'\}$.  If $\lambda_a'$ is disconnected (case
$\square=\parr$ only), take $Z=\error$; otherwise define $Z$ by
replacing $\lambda$ in $X$ with $\{\lambda_a',\lambda_b\}$.
\item $\square_1=\tensor$. Then $\lambda_a\setminus
r_1=\{\lambda_a',\lambda_a''\}$, with $\lambda_a'$ and $\lambda_b'$
connected.  Define $Z$ by replacing $\lambda$ in $X$ with
$\{\lambda_a',\lambda_a'',\lambda_b\}$.
\item $\square_1=\with$.  
Since $r_0$ is ready in $\lambda$, and $\lambda$ is non-empty, $\lambda$ must have a single piece.
Thus $r_1$ is unary, so one of the two slicings obtained by removing $r_1$ is empty.
Since $r_1$ remains unary after erasing $r_0$, we can take $Z=\error$.
\end{itemize}
\item $\square_0=\with$.  
By $r_0/r_1$ symmetry, we need not consider $\square_1=\tensor$.
Let $\lambda\setminus r_0=\{\lambda_a,\lambda_b\}$.
Assume $\lambda_a$ and $\lambda_b$ are connected, or else the result is trivial with $Z=\error$.
We consider subcases for $\square_1$.
\begin{itemize}
\item $\square_1=\with$.  
Since there is no constraint on $\with$-readiness, we can erase the $\with$'s in either order.
However, due to duplication, there are two copies of the second $\with$ to erase.
Let $\Gamma_a$ and $\Gamma_b$ be the sequents of $\lambda_a$ and $\lambda_b$.  
The sequents have copies $r_{1a}$ and $r_{1b}$ of $r_1$, respectively.
We have $\lambda_a\setminus r_{1a}=\{\lambda_{ax},\lambda_{ay}\}$ and $\lambda_b\setminus r_{1b}=\{\lambda_{bx},\lambda_{by}\}$.
Let $\lambda\setminus r_1=\{\lambda_x,\lambda_y\}$.
Analogously, $\lambda_x\setminus r_{0x}=\{\lambda_{xa},\lambda_{xb}\}$
and $\lambda_y\setminus r_{0y}=\{\lambda_{ya},\lambda_{yb}\}$.
Since $\with$-removal merely partitions the pieces of $\lambda$,
we have $\lambda_{ax}=\lambda_{xa}$, and similarly for the other
three.
If any of the four slicings is empty, we take $Z=\error$.  Otherwise,
let $X=X'\cup\{\lambda\}$, where $\lambda\not\in X'$.  Define
$Z=X'\cup\{\lambda_{ax},\lambda_{ay},\lambda_{bx},\lambda_{by}\}$.
Then 
\begin{displaymath}
\begin{array}{c}
X\leadsto_{r_0}Y_0\leadsto_{r_{1a}}X'\cup\{\lambda_{ax},\lambda_{ay},\lambda_b\}\leadsto_{r_{1b}}Z\\[1ex]
X\leadsto_{r_1}Y_1\leadsto_{r_{0x}}X'\cup\{\lambda_{xa},\lambda_{xb},\lambda_y\}\leadsto_{r_{0y}}Z
\end{array}
\end{displaymath}
where the $\leadsto$-subscripts indicate which root is being erased.
\item $\square_1=\plus$ and $\parr$.  The reasoning is analogous to the previous case, though simpler due to less duplication.
\end{itemize}
\item $\square_0=\parr$. By symmetry, we need only consider $\square_1=\parr$ or $\plus$.  This case is trivial, 
since erasing each $r_i$ merely deletes a vertex from a
(sequent)-graph.
It is possible that erasing a $\parr$ can yield a disconnected slicing; in this case we take $Z=\error$.
\item $\square_0=\plus$.  By symmetry, we need only consider $\square_1=\plus$.  This case is trivial.
\end{itemize}
If either $Y_i$ is $\error$ we simply take $Z=\error$.
\end{proof}

Define the \defn{profile} of a cluster as $\langle p,q\rangle$ where
$p$ is the total number of links (summed accross all slicings) plus
the total number of conflict edges, and $q$ is the total number of
connectives (in the underlying sequents).
\begin{theorem}
Erasure $\leadsto$ is terminating (strongly normalising).
\end{theorem}
\begin{proof}
Every $\leadsto$-step
either (a) decreases $p$, while perhaps increasing $q$, or (b)
decreases $q$, without increasing $p$.
\end{proof}
\begin{proposition}
Erasure $\leadsto$ is confluent: if $X\leadsto^*Y_0$ and
$X\leadsto^*Y_1$ then there exists $Z$ such that $Y_0\leadsto^*Z$ and
$Y_1\leadsto^*Z$.
\end{proposition}
\begin{proof}
Cut elimination is locally confluent and terminating, hence confluent
by Newman's lemma \cite{New42}.
\end{proof}
Thus every cluster has a unique $\leadsto$-normal form.
A cluster $X$ is \defn{erasable} if its normal form is empty,
\ie, if $X\leadsto^*\emptyset$.  
A slicing $\lambda$ is erasable if $\{\lambda\}$ is erasable.
\begin{definition}
A \defn{conflict net} is an erasable slicing.
\end{definition}

\subsection{P-time correctness}

The \defn{size} of a coherence space is its number of tokens, and the
size of a sequent is its number of vertices.
\begin{theorem}
Erasability of a slicing $\lambda:L\to\Gamma$ can be checked in p-time
in the sizes of\/ $L$ and\/ $\Gamma$.
\end{theorem}
\begin{proof}
Let $\{\lambda\}=X_0\leadsto X_1\leadsto \ldots \leadsto X_n$ be a
normalisation sequence, let $l$ be the size of $L$, and let $g$ be the
size of $\Gamma$.  Let $m=l^2$, an upper bound on the number of
conflict edges in $L$.  Let $k=l+m$.
Then $n\le k.g$ since whenever a $\leadsto$-step decreases $p$ in the
profile $\langle p,q\rangle$, it increases $q$ to at most $g$, and $p$
remains at most $k$.

It remains to show that determining if a cluster $X$ has a
$\leadsto$-redex --- and if so, executing the $\leadsto$-step --- is p-time
in $l$ and $g$.
First we check to see if every slicing in $X$ is connected, which is
p-time in the total number $v(X)$ of vertices in $X$, and $v(X)\le gl+l$.  (In
the worst case, $X$ has $l$ slicings, each a single link on $\Gamma$.)
If every slicing $\mu\in X$ is connected, we attempt to find a
$\leadsto$-redex.  Erasing axioms is trivial, therefore at worst we
take each final vertex of $X$ in turn, and check for readiness.  Checking
for readiness involves only finding connected components of graphs
($\adjacencygraph{M}$ and $\conflictgraph{\mu}$, where $M$ is
the domain of $\mu$).
\end{proof}

\subsection{Sequentialization}

\begin{theorem}[Sequentialization]
A linking is a conflict net iff it is sequentializable.
\end{theorem}
\begin{proof}
The right-to-left implication is a routine
induction over the interpretation of rules as operations on linkings
(Figure~\ref{fig-translation}).

Conversely, a normalisation sequence
$\{\lambda\}=X_1\leadsto\ldots\leadsto X_n=\emptyset$ produces a proof
rule-by-rule, from bottom-to-top, exactly as in the case of circuit nets
(see the proof of Theorem~\ref{thm-boxless-seq}).  Every
$\leadsto$-step yields one non-permutation rule, plus some permutations.
\end{proof}

\section{Alternative representations of conflict nets}

Translation from a proof to a conflict net is quadratic-time in the
size of the proof (due to the conflict edges).
If we are willing to code slightly more information in the
representation, we can obtain a variant for which translation is
linear time.
A \emph{sum net} 
 collapses all parallel axiom
links to a single link, and labels every axiom link with a formal sum of
monomials.
For example, here are the sum net
representations of the two conflict nets at the bottom of
Figure~\ref{fig-egs}, respectively:
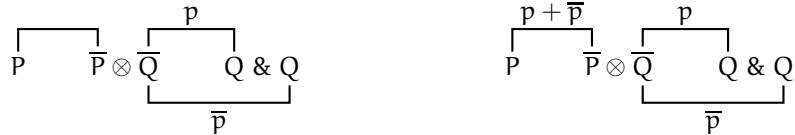
\begin{center}%
\vspace*{5ex}%
\begin{math}\psset{labelsep=1pt}
\renewcommand{\tagwith}[1]{\with}
\newcommand{\gap}{\hspace{5ex}}
\newcommand{\sequent}{\Rnode{P}{\atomOne}\gap\Rnode{P'}{\dual\atomOne}\,\tensor\,\Rnode{Q'}{\dual\atomTwo}\gap\Rnode{Q1}{\atomTwo}\,\tagwith p\,\Rnode{Q2}{\atomTwo}}
\sequent%
\ncbar[angle=90,nodesep=2pt,arm=.23cm]{P}{P'}
\ncbar[angle=90,nodesep=2pt,arm=.23cm]{Q'}{Q1}
\naput[labelsep=1pt]{p}
\ncbar[angle=-90,nodesep=2pt,arm=.23cm]{Q'}{Q2}
\nbput[labelsep=2pt]{\dual p}
\hspace*{16ex}
\sequent%
\ncbar[angle=90,nodesep=2pt,arm=.23cm]{P}{P'}
\naput[labelsep=1pt]{p+\dual p}
\ncbar[angle=90,nodesep=2pt,arm=.23cm]{Q'}{Q1}
\naput[labelsep=1pt]{p}
\ncbar[angle=-90,nodesep=2pt,arm=.23cm]{Q'}{Q2}
\nbput[labelsep=2pt]{\dual p}
\end{math}
\vspace*{5ex}
\end{center}
Girard discusses a relationship between monomials and coherence in
Appendix~A.1.1 of \cite{Gir96}.

A \emph{tree net} is another alternative.  The undirected graph of the
$\conflict$ conflict relation of a proof net is always $P_4$-free
(contractible), thus can be represented by a tree (the so-called
\emph{cotree} associated with a $P_4$-free graph).  For example, here
are the tree net versions of the last two conflict nets in
Figure~\ref{fig-egs}:
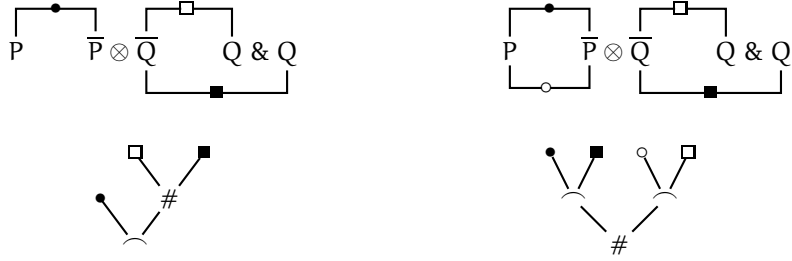
\begin{center}%
\vspace*{5ex}%
\begin{math}\psset{labelsep=1pt}
\renewcommand{\tagwith}[1]{\with}
\newcommand{\gap}{\hspace{5ex}}
\newcommand{\sequent}{\Rnode{P}{\atomOne}\gap\Rnode{P'}{\dual\atomOne}\,\tensor\,\Rnode{Q'}{\dual\atomTwo}\gap\Rnode{Q1}{\atomTwo}\,\tagwith p\,\Rnode{Q2}{\atomTwo}}
\sequent%
\ncbar[angle=90,nodesep=2pt,arm=.3cm]{P}{P'}
\ncput[nodesep=1pt]{\smallbulle}
\ncbar[angle=90,nodesep=2pt,arm=.3cm]{Q'}{Q1}
\ncput*[labelsep=1pt]{\mkern-6.5mu\smallsquar\mkern-6mu}
\ncbar[angle=-90,nodesep=2pt,arm=.3cm]{Q'}{Q2}
\ncput[labelsep=2pt]{\smallblacksquar}
\hspace*{16ex}
\sequent%
\ncbar[angle=90,nodesep=2pt,arm=.3cm]{P}{P'}
\ncput[labelsep=1pt]{\smallbulle}
\ncbar[angle=90,nodesep=2pt,arm=.3cm]{Q'}{Q1}
\ncput*[labelsep=1pt]{\mkern-6.5mu\smallsquar\mkern-6mu}
\ncbar[angle=-90,nodesep=2pt,arm=.3cm]{P}{P'}
\ncput*[labelsep=2pt]{\mkern-7mu\smallcircl\mkern-7mu}
\ncbar[angle=-90,nodesep=2pt,arm=.3cm]{Q'}{Q2}
\ncput[labelsep=-5pt]{\smallblacksquar}
\end{math}

\vspace*{12ex}

\begin{math}
\hspace*{-3ex}
\psset{unit=.9ex,nodesepA=1.3pt,nodesepB=-.5pt}
\rput(0,0){\rnode{stcoh}{\girstcoh}}%
\rput(0,4){%
  \rput(-3,0){\rnode{bulle}{\smallbulle}}%
  \rput(3,0){%
    \rput(0,0){\rnode{conflict}{\conflict}}%
    \rput(0,4){%
       \rput(-3,0){\rnode{squar}{\smallsquar}}%
       \rput(3,0){\rnode{blacksquar}{\smallblacksquar}}%
    }
  }%
}
\ncline[nodesepB=-2pt]{stcoh}{bulle}
\ncline[nodesepB=1.3pt]{stcoh}{conflict}
\ncline{conflict}{squar}
\ncline{conflict}{blacksquar}
\hspace*{38ex}
\rput(0,0){\rnode{conflict}{\conflict}}%
\rput(0,4){%
  \rput(-4,0){%
    \rput(0,0){\rnode{stcoh}{\girstcoh}}%
    \rput(0,4){%
       \rput(-2,0){\rnode{bulle}{\smallbulle}}%
       \rput(2,0){\rnode{blacksquar}{\smallblacksquar}}%
    }
  }%
  \rput(4,0){%
    \rput(0,0){\rnode{stcoh'}{\girstcoh}}%
    \rput(0,4){%
       \rput(-2,0){\rnode{circl}{\smallcircl}}%
       \rput(2,0){\rnode{squar}{\smallsquar}}%
    }
  }%
}
\ncline[nodesepB=.3pt]{conflict}{stcoh}
\ncline[nodesepB=.3pt]{conflict}{stcoh'}
\ncline{stcoh}{bulle}
\ncline{stcoh}{blacksquar}
\ncline{stcoh'}{circl}
\ncline{stcoh'}{squar}
\end{math}
\vspace*{3ex}
\end{center}
This tree on axiom links is obtained readily from a proof, in linear
time: it is the underlying $\tensor$- and $\with$-rule binary tree,
modulo associativity and commutativity, with $\tensor$-rules providing
strict coherence $\girstcoh$ between axioms, and $\with$ providing
conflict (strict incoherence) $\conflict$.

\section{P-time correctness for slice nets, by erasure}\label{sec-slice-ptime}

By using erasure, we prove that the correctness of a slice net $\Lambda$ on $\Gamma$
\cite{HG03,HG05} can be checked in p-time in the number of links in $\Lambda$ and 
the number vertices in $\Gamma$.
Recall that a linking of a slice net is a slicing $\lambda : L\to\Gamma$ with $L$ a non-empty clique.

Let $\Lambda$ be a set of linkings, or \emph{linking-set}, on $\Gamma$.
A link in/of $\Lambda$ is a link in a linking of $\Lambda$ (\ie, a link in $\bigcup\Lambda$).
Define $G(\Lambda,\Gamma)$ as the graph comprising $\Gamma$ and every link in $\Lambda$.
$\Lambda$ is \defn{connected} if it is non-empty and $G(\Lambda,\Gamma)$ is connected.

Let $\Lambda$ be a connected linking on $\Gamma$, and let $r$ be a
root of $\Gamma$, the root of the formula $A_0\square A_1$.
Define $r$ as \defn{ready} if it matches one of the following cases:
\begin{itemize}
\item
$\square=\parr$.
\item
$\square=\with$.
\item
$\square=\plus$ and $r$ is unary: for some $j\in\{0,1\}$ no link in $\Lambda$ has a leaf in the formula $A_j$.
\item
$\square=\tensor$. 
Deleting $r$ disconnects $G$ into two components $G_i$, where $A_i$ is
a formula in $G_i$.  
Let the underlying sequent of $G_i$ be $\Delta_i$.
For each linking $\lambda\in\Lambda$ define $\lambda_i$ as the
restriction of $\lambda$ to $\Delta_i$ (thus $\lambda=\lambda_0\cup
\lambda_1$).  Define $\Lambda_i=\{\lambda_i:\lambda\in\Lambda\}$.
Let $n_i$ be the number of linkings in $\Lambda_i$, and $n$ the number of linkings in $\Lambda$.
Then\footnote{By construction, $n\le n_0\times n_1$ always holds, since we work with \emph{sets} of linkings.} $$n\;\;=\;\;n_0\times n_1.$$
\end{itemize}
When ready, the result $\Lambda\setminus r$ of \defn{erasing} $r$ is:
\begin{itemize}
\item
$\square=\parr$.  $\Lambda_0$ on $\Gamma_0$, where $\Gamma_0$ has $A_0,A_1$ in place of $A_0\parr A_1$.
\item
$\square=\plus$.  $\Lambda_j$ on $\Gamma_j$, where $\Gamma_j$ has $A_j$
in place of $A_0\plus A_1$, according to whether a link of $\Lambda$ has a leaf in $A_j$.
\item
$\square=\with$.  $\Lambda_0$ on $\Gamma_0$ and $\Lambda_1$ on
$\Gamma_1$, where $\Gamma_i$ has $A_i$ in place of $A_0\with A_1$, and
$\Lambda_i$ comprises every linking of $\Lambda$ which has a link with
a leaf in $A_i$.  (Thus $\Lambda=\Lambda_0\cup\Lambda_1$, disjointly.)
\item
$\square=\tensor$.  $\Lambda_0$ on $\Delta_0$ and $\Lambda_1$ on
$\Delta_1$, where $\Delta_i$ and $\Lambda_i$ are as in the definition of
$\tensor$-readiness above.
\end{itemize}
Note that even though $\Lambda$ is connected, a linking-set in
$\lambda\setminus r$ may be disconnected (\eg\ empty).

The following definitions are practically identical to those for erasure of conflict nets.
A \defn{cluster} is either a set of linking-sets or the \defn{error} symbol $\error$.
Define \defn{erasure} $\leadsto$ on clusters as follows.
\begin{itemize}
\item $Y\leadsto\error$ if $Y$ contains a linking-set which is disconnected. (Note: any empty linking-set is disconnected.)
\item $X\cup\{\Lambda\}\,\leadsto\,X\cup(\Lambda\setminus r)$ if $r$ is a ready root of $\Lambda$,
and every linking-set in $X$ is connected.  
Here we assume $\Lambda\not\in X$.
\item $X\cup\{\Lambda\}\,\leadsto\,X$ if $\Lambda$ has a single link, on $P,\dual P$ for some literal $P$ 
(\ie, if $\Lambda$ corresponds to an axiom), and every linking-set of $X$ is connected.  Here we assume $\Lambda\not\in X$.
\end{itemize}
Erasure $\leadsto$ is confluent and termining by the same reasoning as
for conflict nets.
The same reasoning with profiles shows that the path-length to normal
form is polynomial in the number of links $l$ and the number of sequent vertices $g$.
Each form of readiness for a root is clearly p-time checkable.
That erasure coincides with sequentializability is again a routine
induction, as with circuits and conflict linkings.

\section{Cut elimination}\label{monomial-cut-elim}

Cut elimination for conflict nets is work in progress.  The same is
true for monomial nets: the proposal for their cut elimination
sketched in \cite[App.\,A.1.2--3]{Gir96} is ill-defined.
A counter-example is shown below.
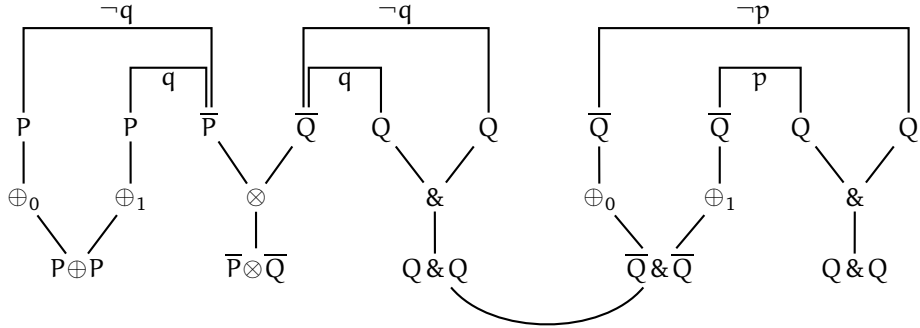
\begin{center}\vspace*{10ex}
\begin{math}\psset{labelsep=1pt,nodesep=2pt}
\begin{array}{ccccccccccccccccccc}
\Rnode{1}{P}
&&
\Rnode{2}{P}
&&
\Rnode{3}{\dual P}
&&
\Rnode{4}{\dual Q}
&&
\Rnode{5}{Q}
&&
\Rnode{6}{Q}
&&
\Rnode{7}{\dual Q}
&&
\Rnode{8}{\dual Q}
&&
\Rnode{9}{Q}
&&
\Rnode{10}{Q}
\ncbar[angle=90,arm=15pt,offsetB=-1pt]{2}{3}\nbput{q}
\ncbar[angle=90,arm=30pt,offsetB=1pt]{1}{3}\naput{\neg q}
\ncbar[angle=90,arm=15pt,offsetB=1pt]{5}{4}\naput{q}
\ncbar[angle=90,arm=30pt,offsetB=-1pt]{6}{4}\nbput{\neg q}
\ncbar[angle=90,arm=15pt]{8}{9}\nbput{p}
\ncbar[angle=90,arm=30pt]{7}{10}\naput{\neg p}
\\[3ex]
\Rnode{a}{\plus_0}
&&\Rnode{b}{\plus_1}
&&&\Rnode{c}{\tensor}
&&&&\Rnode{d}{\with}
&&&\Rnode{e}{\plus_0}
&&\Rnode{f}{\plus_1}
&&&\Rnode{g}{\with}
\ncline{1}{a}
\ncline{2}{b}
\ncline{3}{c}
\ncline{4}{c}
\ncline{5}{d}
\ncline{6}{d}
\ncline{7}{e}
\ncline{8}{f}
\ncline{9}{g}
\ncline{10}{g}
\\[3ex]
&\Rnode{A}{\!\!\!\!P\plus P\!\!\!\!}
&&&&\Rnode{B}{\!\!\!\!\dual P\tensor \dual Q\!\!\!\!}
&&&&\Rnode{C}{\!\!\!\!Q\with Q\!\!\!\!}
&&\;\;\;\;&&\Rnode{D}{\!\!\!\!\dual Q\with \dual Q\!\!\!\!}
&&&&\Rnode{E}{\!\!\!\!Q\with Q\!\!\!\!}
\ncline{a}{A}
\ncline{b}{A}
\ncline{c}{B}
\ncline{d}{C}
\ncline{e}{D}
\ncline{f}{D}
\ncline{g}{E}
\nccurve[angleA=-50,angleB=-130]{C}{D}
\end{array}
\end{math}\vspace*{8ex}
\end{center}
The definition of cut elimination fails to work because spreading is
limited to a single formula: this means that after spreading above the
central $Q\with Q$ with respect to $p$, we do not have a proof
structure (contrary to the claim at the end of A.1.2 in
\cite{Gir96}).
To fix cut elimination, one would at a minimum have to extend
spreading: in the example above, performing something related
to spreading above the left-most formula $P\plus P$.

\section{Relationship with combinatorial proofs}\label{sec-comb}

A \emph{combinatorial proof} \cite{Hug06} is an abstraction notion of
proof net for classical logic \cite{Hug06i}.
A combinatorial proof of a classical formula $A$ is a graph
homomorphism $h:L\to G(A)$ from a partitioned $P_4$-free
(contractible) graph $L$ to a graph $G(A)$ associated with $A$,
satisfying certain conditions.  A combinatorial proof of Peirce's law
$((\dual P\vee Q)\wedge \dual P)\vee P$ is shown below.
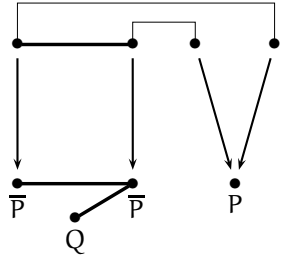
\begin{center}
\vspace*{14ex}
\begin{math}\psset{nodesep=-1pt}
\rnode{p1}{\bullet}
\hspace*{8ex}
\rnode{p2}{\bullet}
\hspace*{7ex}
\rnode{P}{\bullet}
\ncline[linewidth=1.3pt]{p1}{p2}
\nbput[labelsep=2ex]{\rnode{Q}{\bullet}}
\ncline[linewidth=1.3pt]{p2}{Q}
\nput*[labelsep=10ex]{90}{p1}{\rnode{p1'}{\bullet}}
\nput*[labelsep=10ex]{90}{p2}{\rnode{p2'}{\bullet}}
\nput*[offsetA=-15pt,labelsep=10ex]{90}{P}{\rnode{P1'}{\bullet}}
\nput*[offsetA=15pt,labelsep=10ex]{90}{P}{\rnode{P2'}{\bullet}}
\ncline[nodesep=3pt]{->}{p1'}{p1}
\ncline[nodesep=3pt]{->}{p2'}{p2}
\ncline[nodesep=3pt]{->}{P1'}{P}
\ncline[nodesep=3pt]{->}{P2'}{P}
\ncbar[linewidth=.3pt,labelsep=-1pt,angle=90,arm=3ex]{p1'}{P1'}
\ncbar[linewidth=.3pt,labelsep=-1pt,angle=90,arm=1.5ex]{p2'}{P2'}
\ncline[linewidth=1.3pt]{p1'}{p2'}
\nput*[labelsep=2pt]{-90}{p1}{\dual P}
\nput*[labelsep=2pt]{-90}{Q}{Q}
\nput*[labelsep=2pt]{-90}{p2}{\;\dual P}
\nput*[labelsep=2pt]{-90}{P}{P}
\end{math}
\vspace*{5ex}
\end{center}
The partitioned graph $L$ is on top, with four vertices and one
(thick, horizontal) edge, and two two-vertex classes
indicated by (thin) link-style edges.
The graph $G(A)$ is underneath, with four vertices and two edges.  Its
vertices are the literals of $A$, with an edge between literals when
the smallest subformula containing them is a conjunction.
The arrows indicate the graph homomorphism $h$.

The graph homomorphism is required to be a \emph{skew fibration}.
A coherence space map, as in a slicing, is just a relational generalisation of a graph
homomorphism; the skew fibration property corresponds to maximality.
Thus slicings are very closely related to combinatorial proofs.

\small
\bibliographystyle{alpha}
\bibliography{../../bib/main}

\end{document}